\documentclass{lmcs}
\usepackage[utf8]{inputenc}
\pdfoutput=1

\usepackage{lastpage}
\lmcsdoi{18}{3}{15}
\lmcsheading{}{\pageref{LastPage}}{}{}%
{Apr.~05,~2021}{Aug.~02,~2022}{}

\keywords{Constructive Analysis, Exponentiation, Logarithm, Point-free Topology, Geometric Logic, Places}
\usepackage{amsfonts, amsthm, amssymb, amsmath, stackengine, scalerel}
\usepackage{comment}
\usepackage{hyperref}
\usepackage{mathrsfs,array,tikz-cd}
\usepackage{color}
\usepackage[all]{xy}
\usepackage{url}
\usepackage{extpfeil}
\usepackage{verbatim}  
\usetikzlibrary{graphs,decorations.pathmorphing,decorations.markings}

\xyoption{all}

\def\foo{\,\ThisStyle{\ensurestackMath{%
			\bigsqcup\stackengine{-0pt}{\!}{\SavedStyle\!^{\mathord{\uparrow}}}{O}{l}{F}{T}{S}}}\,}
\DeclareMathOperator*{\foobarX}{\foo}
\newcommand\dirsup{\!\foobarX}  

\newextarrow{\xbigtoto}{{20}{20}{20}{20}}
{\bigRelbar\bigRelbar{\bigtwoarrowsleft\rightarrow\rightarrow}}

\newcommand{\Reals}{\mathbb{R}}
\newcommand{\RIdl}{\mathsf{RIdl}}

\newcommand{\Spec}{\mathrm{Spec}}

\newcommand{\Set}{\mathrm{Set}}

\newcommand{\opp}{\mathrm{op}}    
\newcommand{\argu}{\text{(---)}}

\newcommand{\thT}{\mathbb{T}}     
\newcommand{\thU}{\mathbb{U}}     
\newcommand{\baseS}{\mathcal{S}}  
\newcommand{\pt}{\mathsf{pt}}     

\newcommand{\cosimp}[3]{\xymatrix@1{#1 \ar@<.4ex>[r] \ar@<-.4ex>[r] & {\ }#2 \ar@<0.8ex>[r] \ar[r] \ar@<-.8ex>[r] & {\ } #3 \ar@<1.2ex>[r] \ar@<.4ex>[r] \ar@<-.4ex>[r] \ar@<-1.2ex>[r] & \cdots }} 

\newcommand{\equalizer}[2]{\xymatrix@1{#1 \ar@<.4ex>[r] \ar@<-0.4ex>[r] & {\ } #2}}

\newcommand{\adjunction}[4]{\xymatrix@1{#1{\ } \ar@<-0.3ex>[r]_{ {\scriptstyle #2}} & {\ } #3 \ar@<-0.3ex>[l]_{ {\scriptstyle #4}}}} 

\newcommand{\ropen}[1]{[#1)} 
\newcommand{\lopen}[1]{(#1]} 

\theoremstyle{definition}
\newtheorem{discussion}[thm]{Discussion}

\begin{document}

	\title{Point-free Construction of Real Exponentiation}
	\author[M.~Ng]{Ming Ng\lmcsorcid{0000-0002-2170-0247}}	

	\author[S.~Vickers]{Steven Vickers\lmcsorcid{0000-0003-1907-9014}}	

	\address{School of Computer Science,
		The University of Birmingham,
		Birmingham, B15 2TT, UK.}	
	\email{mxn732@bham.ac.uk, s.j.vickers@cs.bham.ac.uk}  
    \subjclass[2020]{Primary 26E40; Secondary 18F10,18F70}

	\begin{abstract}
		We define a point-free construction of real exponentiation and logarithms, i.e.\ we construct the maps $\exp\colon (0, \infty)\times \mathbb{R} \rightarrow \!(0,\infty),\, (x, \zeta) \mapsto x^\zeta$ and $\log\colon (1,\infty)\times (0, \infty) \rightarrow\mathbb{R},\, (b, y) \mapsto \log_b(y)$, and we develop familiar algebraic rules for them.
		The point-free approach is constructive, and defines the points of a space as models of a geometric theory, rather than as elements of a set --- in particular, this allows geometric constructions to be applied to points living in toposes other than $\Set$.
		Our geometric development includes new lifting and gluing techniques in point-free topology, which highlight how properties of $\mathbb{Q}$ determine properties of real exponentiation.

		This work is motivated by our broader research programme of developing a version of adelic geometry via topos theory. In particular, we wish to construct the classifying topos of places of $\mathbb{Q}$, which will provide a geometric perspective into the subtle relationship between $\mathbb{R}$ and $\mathbb{Q}_p$, a question of longstanding number-theoretic interest.
	\end{abstract}
	\maketitle

	\section*{Introduction}

	Our primary objective here is self-explanatory: we wish to develop an account of exponential and logarithmic functions
	\begin{align*}
	\exp\colon (0, \infty)\times \mathbb{R} &\longrightarrow \!(0,\infty)\\
	(\,x\,,\, \zeta\,) &\longmapsto x^\zeta
	\end{align*}
	\begin{align*}
	\log\colon (1,\infty)&\!\times\! (0, \infty) \longrightarrow\mathbb{R}\\
	(\,b\,&, \, y\,) \longmapsto \log_b(y)
	\end{align*}
	that is \emph{geometric} --- that is to say, it is valid over any topos, and moreover is preserved by pullback along geometric morphisms. In the final section (Section~\ref{sec:ExpNT}),
	we give examples of intended applications where the geometricity might be exploited,
	including working towards a topos-theoretic account of places of $\mathbb{Q}$, which will give a geometric perspective into the subtle relationship between $\mathbb{R}$ and $\mathbb{Q}_p$, a question of longstanding number-theoretic interest.

	As is well known, geometricity cannot be achieved using point-set topology, and instead a point-free approach must be taken --- though in our geometric methodology the points will still play a major role.
	Indeed, this philosophy will guide our construction of the exponentiation and logarithm maps on the Dedekind reals, as well as our development of some of their obvious algebraic properties. As far as the authors are aware, this is the first time these maps have been defined in a point-free setting.

	The heuristic behind our construction is simple, and involves building up to real exponentiation in increasing levels of (topological) complexity to get to the general case:\footnote{Convention: In this paper, we shall denote $Q$ to be the set of non-negative rationals and $Q_+$ to be the set of positive rationals. See Section 2 for more details.}
	\begin{itemize}
		\item[] \textbf{Step 1:} Define natural number exponentiation for non-negative rationals:    $x^a$ for $x\in Q$ and $a\in\mathbb{N}$.
		\item[] \textbf{Step 2:} Define natural number exponentiation for non-negative reals: $x^{a}$ for $x\in \ropen{0, \infty}$ and $a\in\mathbb{N}$.
		\item [] \textbf{Step 3:} Define rational exponentiation for non-negative reals: $x^{q}$ for $x\in \ropen{0, \infty}$ and $q\in \mathbb{Q}$.
		\item[] \textbf{Step 4:} Define real exponentiation for positive reals:    $x^{\zeta}$ for $x\in (0,\infty)$ and $\zeta\in \mathbb{R}$.
	\end{itemize}

    \smallskip\noindent
	However, Step 4 presents several geometric issues. For one, working with real exponents creates continuity issues at $x=0$, forcing us to work with positive Dedekind base. Additionally, exponentiation can either be monotone or antitone, depending on whether $x>1$ or $x<1$. These two different cases require individual treatment, which gives rise to a piecewise account of exponentiation, raising further continuity issues. In light of this, we develop some new lifting and gluing techniques for localic spaces (see Lemma~\ref{lem:rationalsliftreals} and Section~\ref{subsec:gluingtechnique}) which allow us to glue these different cases of exponentiation together to obtain a continuous map.

	\paragraph{Outline of Paper} The preliminaries of point-free mathematics and the geometric theory of Dedekind reals are laid out in Section~\ref{sec:PrelimExp}. In Section~\ref{sec:ExpI}, we tackle the basic case of defining rational exponentiation for positive reals, along with many of the familiar algebraic properties of exponentiation. We extend these results in Section~\ref{sec:ExpII} to obtain the general account of real exponentiation for positive reals. As an immediate application of these properties, we also obtain a point-free account of logarithms in Section~\ref{sec:log}. Finally in Section~\ref{sec:ExpNT}, we sketch out a couple other possible point-free approaches to exponentiation, before discussing the relevance of a point-free construction of real exponentiation to constructing the classifying topos of places, and (more broadly) to Arithmetic Geometry and Number Theory.

	\section{Preliminaries in Point-free Topology and Topos Theory}\label{sec:PrelimExp}
The goal of this section is to give a quick introduction to the basic ideas and constructions used in this paper, as well as fixing some important notation and terminology. We shall also introduce two important techniques that we will rely heavily on in this paper --- the first is a Lifting Lemma that allows us to lift certain results from the rationals to the reals, the second is a gluing technique that will allow us to (constructively) justify performing certain kinds of case-splitting of the reals.


	\subsection{Point-Free Topology and Geometric Logic}%
	\label{subsec:PRELIMPointfree}
	The defining feature of a point-free space is that its points are not described as elements of a set (as in a \emph{point-set} space) but as models of a geometric theory.
	This general principle covers the various interpretations of ``point-free'', including locales (using the geometric theory of completely prime filters), and formal topology (which presents the geometric theory directly, with the base as signature and covers as axioms).

	\begin{conv}[Point-free vs. Point-set Topology]\leavevmode
		\begin{itemize}
			\item In this paper, topology and spaces will always be point-free --- unless we explicitly specify that they are ``point-set''.
			\item If $\thT$ denotes a geometric theory, then we write $[\thT]$ for the space of points of $\thT$. 
		\end{itemize}
	\end{conv}

    \noindent
	We emphasise that working ``point-free'' does \emph{not} mean working point\emph{less}ly, i.e.\ without mentioning points at all. In fact we can define both maps and bundles in terms of points, as follows:

	\begin{defi}[Geometricity]\label{def:MapBundle}\leavevmode
		\begin{itemize}
			\item
			A map $f\colon [\thT]\to [\thT']$ is defined by a \emph{geometric} construction of points $f(x)$ of $[\thT']$ out of points $x$ of $[\thT]$.
			\item
			A bundle $p\colon (\Sigma_{x\in [\thT]}[\thU(x)])\to [\thT]$
			is defined by a \emph{geometric} construction of geometric theories $\thU(x)$
			(the \emph{fibre} over $x$) out of arbitrary points $x$ of $[\thT]$.
		\end{itemize}
		Here the bundle space corresponds to a theory that extends $\thT$ with the ingredients of $\thU(x)$: a point is a pair $(x,y)$ where $x$ is a point of $[\thT]$ and $y$ is a point of $[\thU(x)]$. As a map, $p$ acts by model reduction --- it forgets $y$.

	\end{defi}

	The emphasis here is on \emph{geometric:} this refers (categorically) to a mathematics of  colimits and finite limits --- which usefully includes free algebra constructions. The key insight, which goes back to Grothendieck, is that the geometric mathematics has an intrinsic continuity that includes (for instance in a map $f$) the continuity as defined in point-set topology. In fact it far extends it, going to both generalised spaces (toposes) and a fibrewise topology of bundles (as hinted at above).

	For the purposes of this paper, it is not necessary for us to provide the full technical details of this picture. Instead, it suffices for us to introduce the working definitions of the relevant notions. For a more detailed treatment of point-free topology and topos theory, we recommend~\cite{Vi3,Vi4,Vi5,VickersPtfreePtwise} for justification of some of the deeper aspects.

	Let us now describe geometric theories, first in a strict format.
	Note the two-level distinction between formulae and sequents, which prevents implication ($\Rightarrow$) and universal quantification ($\forall$) from being nested.

	\begin{defi}\label{def:geomlogic} Let $\Sigma$ be a first-order signature of sorts and symbols. Then over $\Sigma$, we define the following:
		\begin{itemize}
			\item
			A \emph{geometric formula} is a logical formula built up from the symbols in $\Sigma$ and a context of finitely many free variables,
			using truth $\top$, equality $=$, finite conjunctions $\land$, arbitrary (possibly infinite) disjunctions $\bigvee$, and $\exists$.
			\item
			A \emph{geometric sequent} is an expression of the form
			$\forall xyz\ldots(\phi\Rightarrow \psi)$,
			where $\phi$ and $\psi$ are geometric formulae in the same context $\{x,y,z,\ldots\}$.
			\item
			A \emph{geometric theory} over $\Sigma$ is a set $\mathbb{T}$ of geometric sequents, the \emph{axioms} of the theory.
		\end{itemize}
	\end{defi}

    \noindent
	One can see a connection with topology in the match between the disjunctions and finite conjunctions in the logic, and the unions and finite intersections of open sets in topology. However, the real importance lies in the fact that the logic can be expressed (categorically) within the geometric \emph{mathematics} of colimits and finite limits.

	For geometric theories we don't essentially extend their expressive power if we allow the signature to include sorts (and their structure) that are constructed geometrically out of others.
	This is extremely useful in practice, and we shall see it in the theory of real numbers (Definition~\ref{def:DedekindReals}),
	where one of the sorts is the rationals $\mathbb{Q}$.
	The technique is described informally in~\cite{Vi4} and with more detailed justification in~\cite{Vi5}.

	\begin{defi}\label{def:model} Let $\mathbb{T}$ be a geometric theory. If $M$ (which may be internal to some category) is a mathematical structure for $\Sigma$ that satisfies the axioms of $\mathbb{T}$, we call $M$ a  \emph{model of $\mathbb{T}$}. The possible universes inside which models of $\mathbb{T}$ may live are \emph{toposes}.
	\end{defi}

	\begin{conv}\label{rem:topos!} For this paper, the unqualified term ``topos'' will always means Grothendieck topos, and $\mathfrak{Top}$ is the 2-category of Grothendieck toposes.
		While a Grothendieck topos can be characterized as a category of sheaves over a Grothendieck site $(\mathcal{C},J)$, we shall not need that description.
		Our geometric methods require us to look at the \emph{points} of a topos
		(the models of a theory it classifies) rather than its \emph{sheaves}\footnote{In fact, the site is just a form of geometric theory (of continuous flat functors) for which there is a convenient concrete description of the sheaves.}. Our work remains valid relative to any elementary topos $\mathcal{S}$ with natural number object (nno), in which case $\mathfrak{Top}$ means the 2-category $\mathfrak{BTop}/\mathcal{S}$ of bounded $\mathcal{S}$-toposes.
	\end{conv}

	\begin{exa} The usual algebraic laws of commutative rings can be formulated as geometric axioms, giving us a geometric theory. For instance, once we include $+,\cdot$ as our logical symbols in $\Sigma$, we may express the  distributivity law of rings as:
		\[
		\forall xyz (\top \Rightarrow x\cdot (y+z)=(x\cdot y) + (x\cdot z)).
		\]
		Examples of models of this theory would be commutative rings --- e.g.\ the rationals $\mathbb{Q}$, the finite prime fields $\mathbb{F}_p$, etc.
	\end{exa}

	\begin{defi}\label{def:proptheory}
		Let $\mathbb{T}$ be a geometric theory. If its signature has no sorts [so there can be no variables or terms, nor existential quantification],
		then we call $\mathbb{T}$ a \emph{propositional theory}.
	\end{defi}

	That is the definition for geometric theories in the strict form of Definition~\ref{def:geomlogic}.
	However, when we allow constructed sorts, then we can also say that a theory is \emph{essentially propositional} if all its sorts can be constructed `out of nothing' --- a prime example being the geometric theory of Dedekind reals (Definition~\ref{def:DedekindReals}).

	In the case of a propositional theory, there is a close connection with point-set topology:
	the propositions supply a point-set topology on the set $\pt[\thT]$ of models in $\Set$.
	The connection is well worked out in locale theory~\cite{StoneSp}, and we shall call a space \emph{localic} if it comes from an essentially propositional theory.
	Most of the spaces appearing in this paper are localic.

	In this sense, the localic spaces are ``ungeneralised''.
	Though the point-free approach is mathematically different, and not exactly equivalent, its objects of study are conceptually the same topological spaces as were originally studied point-set\footnote{A remark: close though the connection is, the construction $\pt$ is not geometric. As soon as one uses it, one loses the continuity inherent in geometric reasoning.}.

	For a general geometric theory, however, this cannot be done --- the opens (propositions) are not enough.
	Thus they are a completely new kind of space, without precedent in point-set topology.
	Using the language of Grothendieck, these are \emph{generalised} spaces.

	Referring back to Definition~\ref{def:MapBundle}, it is already notable for maps that the geometric reasoning dispenses with the need for continuity proofs.
	For bundles we begin to see a phenomenon that is scarcely definable in point-set topology:
	a bundle is a continuously indexed family of spaces.
	An important fact is that geometric constructions on bundles work fibrewise.

	\begin{propC}[{\cite[\S 8 - 9]{VickersPtfreePtwise}}]\label{prop:fibres}
		Let $p'\colon (\Sigma_{x'\in [\thT']}[\thU(x')])
		\to[\thT']$
		be a bundle, and $f\colon [\thT]\to[\thT']$ a map.
		Then the following diagram of spaces is a pullback.
		The top map takes $(x,y)$ to $(f(x),y)$.
		\[
		\begin{tikzcd}
		\Sigma_{x\in [\thT]}[\thU(f(x))]
		\ar[r]
		\ar[d,"p"]
		& \Sigma_{x'\in [\thT']}[\thU(x')]
		\ar[d, "p'"]
		\\
		{[\thT]}
		\ar[r, "f"]
		& {[\thT']}
		\end{tikzcd}
		\]
	\end{propC}

	If $\thT'$ is the empty theory (= no sorts or symbols, no axioms) then $[\thT']$ is $\textbf{1}$, the 1-point space, and $\thU$ is a plain theory (= no $x'$ to depend on) and the pullback is $[\thT]\times[\thU]$.
	This has an interesting methodological consequence:

	\begin{conv}[``Fixing $x$'']\label{conv:fixingx}  Suppose we wish to construct a map with multiple arguments, such as
		\[
		f\colon[\thT]\times [\thU]\rightarrow [\thU'].
		\]
		To do this, we shall often say ``fix $x\in [\thT]$'' and then, by the usual process, construct a map
		\[
		f_{x}\colon[\thU]\rightarrow [\thU']
		\text{.}
		\]
		The declaration ``fix $x\in [\thT]$'' means that we are working over $[\thT]$ (technically, in the topos of sheaves $\baseS[\thT]$), so that $[\thU]$ and $[\thU']$ are transported to their products with $[\thT]$, so we are actually defining a commutative triangle as follows ---
		but that is equivalent to the $f$ we wanted.
		\[
		\begin{tikzcd}
		{[\thT]\times[\thU]}
		\ar[rr, "{\langle p, f\rangle}"]
		\ar[dr, swap, "p"]
		&& {[\thT]\times[\thU']}
		\ar[dl, "p"]
		\\
		& {[\thT]}
		\end{tikzcd}
		\]
		The reader may notice that we are doing is reminiscent of dependent type theory, except cast in a topos-theoretic language.
	\end{conv}

	\subsection{Brief explanation of the topos theory}%
	\label{subsec:toposes}
	So far, we have not tried to explain what $[\thT]$,
	the space of models of $\thT$, actually is mathematically.
	This is not trivial, as it is supposed to be the space of \emph{all} models, in all suitable categories --- let us say Grothendieck toposes.

	The solution, a fundamental part of topos theory, is to form a category $\baseS[\thT]$, the \emph{classifying topos},
	which is ``the geometric mathematics freely generated by the generic (or universal) model $U_{\thT}$ of $\thT$''. To elaborate:
	\begin{itemize}
		\item ``A category $\baseS[\thT]$'': The objects of $\baseS[\thT]$ are the \emph{sheaves} over $[\thT]$.
		\item ``Generic'' means that it has no properties other than being a model of $\thT$, so this is really a syntactic construction. Such a generic model exists for any geometric theory.
		\item
		``Freely generated'' means that for any model $M$ in some other topos $\baseS[\thT']$, there is a functor $\baseS[\thT]\to\baseS[\thT']$, unique up to isomorphism, that takes $U_{\thT}$ to $M$ and preserves colimits and finite limits (and hence all geometric constructions).
	\end{itemize}

	\begin{obs}[Conservativity of the generic model]\label{obs:conserv}
    This definition means that any geometric construction applied to the generic point $U_{\thT}$ can be transported along the relevant functor taking $U_{\thT}$ to any other model $M$. In particular, for any geometric sequent $\sigma$, the sequent $\sigma$ holds for $U_\mathbb{T}$ iff it holds for any $\mathbb{T}$-model in any topos.
	\end{obs}

	Let's unpack this in the context of point-free topology. Consider a \emph{map}
	$f\colon[\thT]\to[\thT']$, as described in Definition~\ref{def:MapBundle}.
	To define it, we declare ``let $x$ be a point of $[\thT]$'', and then work geometrically to construct a point $f(x)$ of $[\thT']$.
	In the particular case where $x$ is the generic point $U_{\thT}$, we get a geometric construction of $f(U_{\thT})$ in $\baseS[\thT]$.
	Following Observation~\ref{obs:conserv}, by transport along those functors, that generic construction suffices to describe all the instances for more specific points.
	We thus see that the generic point $U_{\thT}$ plays the role of \emph{formal parameter} $x$ in the definition of $f(x)$, and actual parameters are substituted by transporting constructions along the functors.

	In short, a map $f\colon[\thT]\to[\thT']$ defines a point $f(U_{\thT})$ of $[\thT']$,
	constructed geometrically in $\baseS[\thT]$.
	But that is in turn equivalent to a functor $f^\ast\colon \baseS[\thT']\to\baseS[\thT]$ --- note the reversal of direction --- that preserves colimits and finite limits, and takes $U_{\thT'}$ to $f(U_{\thT})$.
	From preservation of colimits we can get a right adjoint $f_\ast\colon\baseS[\thT]\to\baseS[\thT']$,
	and we have arrived at the usual definition of geometric morphism.
	Our maps are thus geometric morphisms.

	The traditional view in topos theory is that the category $\baseS[\thT]$ ``is'' the generalised space of models of $\thT$. However, we shall distinguish them notationally to make it clear what morphisms are being used,
	as above with $f$, $f^\ast$ and $f_\ast$.

	For propositional theories there is a junior version of these ideas using the language of posets, which will be helpful for our calculations.
	Instead of looking at all geometric constructions on sets, it just considers those on truth values, i.e.\ subsets of 1.

	\begin{defi}\label{def:frame}
		A \emph{frame} is a complete lattice $A$ possessing all small joins $\bigvee$ and all finite meets $\land$, such that the following distributivity law holds
		\[a\land \bigvee S = \bigvee \{a\land b | b\in S\}\]
		where $a\in A, S\subseteq A$.

		A \emph{frame homomorphism} is a function between frames that preserves arbitrary joins and finite meets. Frames and frame homomorphisms form the category \textbf{Frm}.
	\end{defi}

	For a theory $\thT$ we have a frame $\Omega[\thT]$ of \emph{opens} of the associated space [more precisely, $\Omega[\thT]$ is the Lindenbaum algebra of $\thT$, i.e.\ the set of geometric formulae modulo equivalence provable from $\thT$].
	It is the ``propositional geometric logic freely generated by the generic model of $\thT$'' in a way similar to what we said for $\baseS[\thT]$.

	\begin{defi}\label{def:localespace} We define the category \textbf{Spaces} to be the dual of \textbf{Frm},
		i.e.\ \textbf{Spaces} = $\textbf{Frm}^{\opp}$. We shall refer to the objects in \textbf{Spaces} as \emph{localic spaces}, or as \emph{(ungeneralised) spaces}.
	\end{defi}

	Our choice of terminology in Definition~\ref{def:localespace} is justified by the following proposition (for more details, see~\cite{J2,Vi4}):

	\begin{prop}
		Let $\mathcal{E}_L$ be the category of sheaves constructed over the localic space $L$.
		Then $\mathcal{E}_L$ is a topos whose points belong to the
		propositional theory $\mathbb{T}$ of completely prime filters of the associated frame. Conversely, if $\mathbb{T}$ is a propositional theory, then $\baseS[\mathbb{T}]\simeq \mathcal{E}_{\Omega[\mathbb{T}]}$, where $\mathcal{E}_{\Omega[\mathbb{T}]}$ is the category of sheaves constructed over the frame of opens $\Omega[\mathbb{T}]$. Furthermore, any geometric morphism between $\baseS[\mathbb{T}]\rightarrow\baseS[\mathbb{T}']$ corresponds to a localic morphism in \textbf{Spaces}. \end{prop}

	\subsection{The Geometric Theory of Reals}\label{subsec:PRELIMDedekindReals} There are two important different types of reals in point-free topology:  Dedekind reals and the so-called `one-sided reals'. Much of the technical analysis of this paper involves exploiting the relationship between these two notions of reals, as well as their specific characteristics.

	We start with definitions. Denote by $\mathbb{Q}$ the set of rationals, by $Q_+$ the positive rationals, and by $Q$ the non-negative rationals. We denote by $\mathbb{R}$ the space of points of the geometric theory of Dedekind reals, which we explicitly define in the following:

	\begin{defi}\label{def:DedekindReals} The geometric theory of Dedekind Reals, with space $\mathbb{R}$, comprises two relations $L,R\subset \mathbb{Q}$ which satisfy the following axioms:
		\begin{enumerate}
			\item $\exists r\in\mathbb{Q}. R (r)$ \qquad\qquad\qquad\qquad\qquad\qquad\qquad\,\, (Right Inhabitedness)
			\item $\forall r\in\mathbb{Q}.\big(R(r)\Leftrightarrow \exists r'\in\mathbb{Q}.(r'<r\land R(r'))\big)$ \qquad\!\!\! ($\Leftarrow$ Upward closure; $\Rightarrow$ Roundedness)
			\item $\exists q\in \mathbb{Q}. L (q)$ \qquad\qquad\qquad\qquad\qquad\qquad\qquad\,\, (Left Inhabitedness)
			\item $\forall q\in\mathbb{Q}.\big(L(q)\Leftrightarrow \exists q'\in\mathbb{Q}.(q'>q\land L(q'))\big)$ \qquad ($\Leftarrow$ Downward closure; $\Rightarrow$ Roundedness)
			\item $\forall q,r\in\mathbb{Q}.(L(q)\land R(r)\Rightarrow q<r)$ \qquad\qquad\qquad (Separatedness)
			\item $\forall q,r\in\mathbb{Q}.(q<r\Rightarrow L(q)\lor R(r))$ \qquad\qquad\qquad (Locatedness)
		\end{enumerate}
		The two relations $L,R$ correspond to the left and right Dedekind sections of a real number.
	\end{defi}

	\begin{conv}\label{conv:simplification1} We shall often denote a point of $\mathbb{R}$ as $x$, instead of explicitly writing out the pair of relations representing it: $(L_x,R_x)$. We will also use $q < x$ to mean $L_x(q)$ and use $x < r$ to mean $R_x(r)$.
	\end{conv}

	\begin{rem}%
		\label{rem:locAxiom}It is known (e.g.\ see~\cite{MV}) that Axiom (6) is equivalent to the following axiom:
		\[\forall \epsilon\in \mathbb{Q}.\big(\epsilon>0\Rightarrow \exists q,r\in\mathbb{Q}.( L(q)\land R(r)\land r-q<\epsilon)\big).\]
	\end{rem}

	The reader will have observed that we used the rationals $\mathbb{Q}$ as our sort for the theory of $\mathbb{R}$. As mentioned before, $\mathbb{Q}$ can be constructed geometrically `out of nothing' along with the usual arithmetic structure, including a decidable strict order $<$ on $\mathbb{Q}$. Not only does this mean that the theory of Dedekind reals is an essentially propositional theory (and thus $\mathbb{R}$ is a localic space), the strict order $<$ on $\mathbb{Q}$  can also be leveraged to give us various (constructive) ways of comparing two points with each other, be they in $\mathbb{Q}$ or $\mathbb{R}$:

	\begin{fact}[Archimedean Property]\label{fact:archproperty} Given two non-negative rationals $x,y\in Q$, where $x>0$ there exists at least one natural number $N\in\mathbb{N}$ such that $Nx>y$, or equivalently $x>\frac{y}{N}$.
	\end{fact}

	\begin{defi}\label{def:apartness}
		Given two Dedekind reals $x$ and $y$, we denote:
		\begin{itemize}
			\item  $x<y$ if there exists some rational number $q$ such that $x<q<y$. In particular, this defines a strict order on $\mathbb{R}$, and the space of all pairs of Dedekinds satisfying
			$x<y$ defines an open subspace of $\mathbb{R}\times\mathbb{R}$.
			\item $x\geq y$ if $(x,y)$ belongs to the closed complement of  $<$ in $\mathbb{R}\times\mathbb{R}$.
		\end{itemize}
	\end{defi}

	\begin{rem}[Equality of Dedekinds]\label{rem:equalDedekind} Syntactically, two models of a propositional theory are isomorphic if they satisfy exactly the same propositions. Consequently,
		in the case of $\mathbb{R}$, this means that $x=y$ if the following condition holds: for any $q\in\mathbb{Q}$, we have that $q<x$ iff $q<y$ and $x<q$ iff $y<q$.
	\end{rem}

	\begin{discussion}[Topology and Decidability]\label{disc:undecidableorder} Definition~\ref{def:apartness} says that there is a sense in which the strict order $<$ relation on  $\mathbb{Q}$ lifts to yield another relation on $\mathbb{R}$. However,
		there is an issue of decidability here, i.e.\ of whether an open has an \emph{open} Boolean complement (so the open is clopen). In particular,
		$<$ is decidable on the rationals (where, because the rationals are discrete, open subspaces are just subsets), but not on the reals.
	\end{discussion}

	There are also two main classes of spaces/geometric theories closely related to $\mathbb{R}$ that will be of interest to us in this paper. The first important class are the subspaces of $\mathbb{R}$:

	\begin{defi} A \emph{subspace} $[\mathbb{T}']$ of another space $[\mathbb{T}]$ is a space that can be presented with a theory $\mathbb{T}'$ being $\mathbb{T}$ extended with additional axioms. In particular, we define the following subspaces of $\mathbb{R}$:
		\begin{enumerate}[label={(\roman*)}]
			\item Denote $(0,\infty)$ to be the open subspace of positive Dedekinds: this is the subspace of $\mathbb{R}$ satisfying the axiom `$\top\Rightarrow L(0)$'.
			Likewise for $q$ rational, denote $(q,\infty)$ and $(-\infty, q)$  for the subspaces satisfying $\top\Rightarrow L(q)$ and $\top\Rightarrow R(q)$ respectively.
			\item Denote $\ropen{0, \infty}$ to be the closed complement of $(-\infty, 0)$,
			satisfying $R(0)\Rightarrow \bot$.
			Using the axioms of $\mathbb{R}$ we see that this is equivalent to $\forall q\in\mathbb{Q}.\, \left(q<0\Rightarrow L(q)\right)$.
			Similarly we write $\ropen{q,\infty}$ and $\lopen{-\infty, q}$.


			\item Finally,
			we extend the notation an obvious way. For example, $\lopen{0,1}=\lopen{-\infty,1}\wedge(0,\infty)$ has the axioms of both $\lopen{-\infty,1}$ and $(0,\infty)$.
		\end{enumerate}
	\end{defi}

    \noindent
	The second important class of spaces related to $\mathbb{R}$ are the one-sided reals:

	\begin{defi}\label{def:1sidedreals} Recall the axioms defining the Dedekind reals in Definition~\ref{def:DedekindReals}. Then:
		\begin{enumerate}[label={(\roman*)}]
			\item The \emph{upper reals} is a space whose points just satisfy Axiom (2).
			\item The \emph{lower reals} is a space whose points just satisfy Axiom (4).
		\end{enumerate}
		Note that this allows the upper (resp.\ lower) reals to be empty, which correspond to $\infty$ (resp.\ $-\infty$). We can exclude these cases by using Axiom (1) (resp.\ Axiom (3)). Informally, an inhabited upper real (resp.\ lower real) approximates a number from above (resp.\ below), whereas a Dedekind real approximates the number from both directions.
	\end{defi}

	\begin{conv}\label{conv:simplification2}Extending Convention~\ref{conv:simplification1}, given an upper real $x$ (resp.\ a lower real $x$), we often write $x<q$ (resp.\ $q<x$) to mean that $q$ belongs to the subset of rationals constituting $x$. We shall also often refer to the one-sided reals as just the `one-sideds'.
	\end{conv}

	\begin{conv}\label{conv:specorder} The one-sided reals are spaces with the corresponding Scott topologies: for lower reals, $x\sqsubseteq y$ iff $x\leq y$ whereas for upper reals $x\sqsubseteq y$ iff $x\geq y$. Observe that the specialisation order $\sqsubseteq$ for the lower reals agrees with the numerical order, whereas for the upper reals it is the opposite. To reflect this, we shall use arrows on top of the spaces to show the direction of the refinement under their respective specialisation orders. For instance, consider the space $(0,\infty)$ --- we then denote the corresponding space of lower reals as $\overrightarrow{\lopen{0,\infty}}$ and the corresponding space of upper reals as $\overleftarrow{\ropen{0, \infty}}$. In particular, note that we've had to include $\infty$ (resp.\ $0$) in the relevant space of lower reals (resp.\ upper reals) since they must contain arbitrary directed joins with respect to  $\sqsubseteq$.
	\end{conv}

	\begin{fact} There exist natural maps
		\begin{align*}
		L\colon \mathbb{R} &\longrightarrow \overrightarrow{\lopen{-\infty,\infty}}\\
		R\colon \mathbb{R} &\longrightarrow \overleftarrow{\ropen{-\infty,\infty}}
		\end{align*}
		where given a Dedekind real $x=(L_x,R_x)$, $L$ sends $x\mapsto L_x$ and $R$ sends $x\mapsto R_x$.
	\end{fact}

	The one-sided reals occupy a kind of computational sweet spot in our point-free calculations. On the one hand, they correspond more closely to our intuitive notion of a `real number' compared to the mere set of $\mathbb{Q}$. On the other hand, unlike the Dedekinds, the one-sideds can also be viewed as honest \emph{subsets} of $\mathbb{Q}$. The upshot is that there is often a direct sense in which properties can be lifted from the rationals to the one-sideds, so long as they respect the order relation. This is a particularly powerful insight once we realise that a Dedekind real is entirely determined by its left or right sections. Hence, in order to show that certain identities hold on the level of the Dedekinds, it is often enough to show that they hold for either their left or right Dedekind sections, which is typically easier.

	We make these claims precise in the following two lemmas:

	\begin{lem}\label{lem:check1sided} The following are equivalent for Dedekinds $x,y$:
		\begin{enumerate}
			\item $x\leq y$
			\item $L_x\sqsubseteq L_y$
			\item $R_x\sqsupseteq R_y$
		\end{enumerate}
		\begin{proof} $(1)\Rightarrow (2)$:  Suppose $q<L_x$. By roundedness, we can find $q'\in\mathbb{Q}$ such that $q<q'<L_x$. By locatedness, either $q<L_y$ or $ R_y<q'$. If $q<L_y$, then done; else if $R_y<q'$, then $y<x$, contradicting (1).

			$(2)\Rightarrow (1)$: Suppose for contradiction that there exists some $q$ such that $y<q<x$. This means there exists $q\in L_x$ and $q\notin L_y$, contradicting (2). Since (2) implies that $y<x$ does not hold, and $x\leq y$ is the closed complement of $y<x$, this implies (1).

			$(1)\Leftrightarrow (3)$: Analogous to above.
		\end{proof}
	\end{lem}

	By symmetry, we obtain the following corollary, which refines Remark~\ref{rem:equalDedekind}:

	\begin{cor}\label{cor:check1sided} The following are equivalent for Dedekinds $x,y$:
		\begin{enumerate}
			\item $x=y$
			\item $L_x=L_y$
			\item $R_x=R_y$.
		\end{enumerate}
	\end{cor}

    \noindent
	Next, we prove the Lifting Lemma.
	This lemma provides conditions under which a map $f$ on rationals can be lifted to the one-sided reals along a canonical map $\phi$, such as:
	\begin{align*}
	\phi\colon \mathbb{Q} &\longrightarrow \overrightarrow{\lopen{-\infty,\infty}}\\
	q &\longmapsto I_q:=\{q'\in \mathbb{Q}| q'<q\}
	\end{align*}
	that sends each rational $q$ to its one-sided representative. Following~\cite{Smyth,Vi8}, we isolate the main features of this set-up in the following definition:

	\begin{defi}\label{def:Rstructure} Let $\phi\colon Y\rightarrow Z$ be a map in \textbf{Spaces}, i.e.\ the category of localic spaces.
		\begin{enumerate}[label={(\roman*)}]
			\item We call $\phi$ a \textit{surjection} if it is an epimorphism in \textbf{Spaces}. These maps are most simply characterised by the inverse image function $\Omega \phi$ being one to one\footnote{Note: in contrast to point-set topology, this does \emph{not} imply that every point in the codomain has a preimage in the domain (as is clear in the previous example).}.
			\item Consider $(Y,<)$ where $Y$ is a set equipped with a dense transitive order\footnote{Note: we do not require the order to be strict. For more details, see~\cite{Vi8}.} $<$. An \textit{ideal} in $Y$ is a subset $I\subseteq Y$ that is downward-closed and contains an upper bound for each of its finite subsets. In particular, if the set $I_q:=\{q'\in Y| q'<q\}$ is an ideal for all $q\in Y$, then we call $(Y,<)$ an \textit{$R$-structure}.
			\item The \textit{rounded ideal completion} of an $R$-structure $(Y,<)$ is the space $\RIdl(Y,<)$ of all ideals of $Y$.
			The specialisation order is then the partial order by inclusion.
		\end{enumerate}
	\end{defi}

    \noindent
	We can now state the Lifting Lemma in full:

	\begin{lem}[Lifting Lemma]\label{lem:rationalsliftreals} We work in the category \textbf{Spaces}. Given an $R$-structure $(Y,<)$, there exists a surjection $\phi\colon Y\rightarrow \RIdl(Y,<)$, composition with which induces an equivalence between:
		\begin{itemize}
			\item Maps $\overline{f}\colon \RIdl(Y,<)\rightarrow X$
			\item Maps $f\colon Y\rightarrow X$ satisfying the following two lifting conditions for all $q,q'\in Y$:
			\begin{enumerate}[label={(\roman*)}]
				\item \emph{\textbf{(Monotonicity)}}  $q'<q \implies f(q')\sqsubseteq f(q)$.
				\item \emph{\textbf{(Continuity)}} $f(q)=\displaystyle\dirsup_{q'
					<q}f(q')$.
			\end{enumerate}
		\end{itemize}
		where $X$ is any space and  $\sqsubseteq$ is its specalisation order.

	\end{lem}
	\begin{proof} We proceed in stages.

\paragraph{Step 0: Set-up.} We define $\phi\colon Y\rightarrow \RIdl(Y,<)$ as sending $q\mapsto I_q:=\{q'\in Y|q'<q\}$. Since $(Y,<)$ is an $R$-structure, $I_q$ is an ideal by definition, and so $\phi$ is well-defined. In particular, observe that $I_q=\dirsup_{q'<q}I_{q'}$. Further, using the explicit description of the opens in~\cite{Vi8}, one easily verifies that $\phi$ is indeed an epimorphism.

\paragraph{Step 1: Transforming a given map} Suppose we have a map $f\colon Y\rightarrow X$ satisfying the monotonicity and continuity conditions. Since $I\subseteq Y$ for any $I\in\RIdl(Y,<)$, we thus define a map $\overline{f}\colon \RIdl(Y,<)\rightarrow X$ as sending $I\mapsto\dirsup_{q\in I}f(q)$.

		Conversely, suppose we have a map $\overline{f}\colon \RIdl(Y,<)\rightarrow X$. We then define $f\colon Y\rightarrow X$ as sending $q\mapsto \overline{f}(I_q)$ That this definition of $f$ also satisfies the required properties is a consequence of the fact that $I_q=\dirsup_{q'<q}I_{q'}$. For \textbf{monotonicity}, observe that $I_q=\dirsup_{q'<q}I_{q'}$  means that $q'<q\implies I_{q'}\sqsubseteq I_{q}$. Hence, since $\overline{f}(I_{q'}\sqsubseteq I_{q})$ induces $\overline{f}(I_{q'})\sqsubseteq \overline{f}(I_{q})$ by functoriality, this gives $q'<q\implies f(q')\sqsubseteq f(q)$, as desired.
		\textbf{Continuity} follows immediately from the fact that maps preserve directed joins, and so
		$I_q=\dirsup_{q'<q}I_{q'}$ yields the identity $\overline{f}(I_{q})=\dirsup_{q'<q}\,\overline{f}(I_{q'})$.

\paragraph{Step 2: Proving equivalence} We need to show that the transformation of maps in Step 1 are inverses. In one direction, suppose we have $f\colon Y\rightarrow X$ satisfying the monotonicity and continuity conditions, and define $g(q):=\overline{f}(I_q)$. Recalling that maps preserve directed joins, since $\overline{f}(I_q)=\dirsup_{q'<q}f(q')=f(q)$, it follows that $f=g$. In the other direction, given a map $\overline{f}\colon Z\rightarrow X$, define  $\overline{g}(I):=\dirsup_{q\in I}\,\overline{f}(I_q)$ as well as $f(q):=\overline{f}(I_q)$. By construction, it is clear $\overline{f}\circ \phi=f=\overline{g}\circ \phi$. Since $\phi$ is epi, this yields $\overline{f}\circ\phi=\overline{g}\circ\phi\!\implies\! \overline{f}=\overline{g}$, and we are done.
	\end{proof}

	To illustrate the power of the Lifting Lemma, suppose $Y$ is some space of rationals and $\RIdl(Y,<)$ is some space of one-sided reals, connected via an appropriate surjection $\phi\colon Y\rightarrow \RIdl(Y,<)$. The Lifting Lemma then informally says: in order to prove equality of maps on the one-sided reals, it typically suffices to prove this for the rationals.

	\begin{rem}[Instances of the Lifting Lemma]\label{rem:liftingvariation} There are many spaces whose points are rationals as well as many spaces whose points are one-sided reals --- consider, for instance, the various different subsets/subspaces of $\mathbb{Q}$ and $\overrightarrow{[-\infty,\infty]}$ respectively. A natural question: how do we know \textit{which} spaces of one-sided reals correspond to \textit{which} space of rationals as its rounded ideal completion?

		First, an important subtlety. The order $<$ of an $R$-structure need not be strict --- only transitive and dense. Indeed, in many cases, a space of one-sided reals is obtained as a rounded ideal completion $\RIdl(Y,<)$ where $Y$ is some subset of $\mathbb{Q}$, and $<$ is the same as the standard order $<$ on $\mathbb{Q}$, except perhaps reversed or modified to permit edge cases. For this paper, the most relevant examples are:
		\begin{itemize}
			\item $\overrightarrow{\lopen{0,\infty}}\cong \RIdl(Q_+,<)$ and $\overleftarrow{\ropen{0, \infty}}\cong \RIdl(Q_+,>)$ where $<$ is the standard strict order;
			\item $\overrightarrow{[0,\infty]}\cong\RIdl(Q,<)$, where $Q$ denotes the \textit{non-negative} rationals, and so we modify $<$ to allow the relation $0<0$;
			\item $\overrightarrow{[-\infty,\infty]}\cong \RIdl(\mathbb{Q}\cup\{-\infty\},<)$, where we add an additional $-\infty$ symbol to $\mathbb{Q}$ and modify $<$ to allow $-\infty<-\infty$. Similarly, we have $\overleftarrow{[-\infty,\infty]}\cong \RIdl(\mathbb{Q}\cup\{\infty\},>)$.
		\end{itemize}
		To avoid burdening this paper with too much theory, we will not prove these isomorphisms and shall leave the technical details as a dark grey box, if not a black one. A more thorough discussion of the relevant ideas\footnote{We remark that what we call a `dense order' here is referred to as an `interpolative order' in~\cite{Vi8}.} can be found in~\cite{Smyth,Vi8}.

	\end{rem}

	In the remainder of this subsection, we define some basic arithmetic operations on the reals (Dedekind and one-sided), before collecting some familiar arithmetic facts on how these operations interact. More details (including the proofs, which we have omitted) can be found in~\cite{R}. We start with addition and subtraction:

	\begin{defi} Reals $x$ and $y$ can be added by the rules
		\[q<x+y\Leftrightarrow \exists s<x \land \exists r<y. (q\leq s+r)\]
		\[q>x+y\Leftrightarrow \exists s>x \land \exists r>y. (q\geq s+r)\]
		where $q,s,r\in\mathbb{Q}$.
	\end{defi}

	\begin{defi} Reals $x$ can be negated by the rules
		\[-x<q\Leftrightarrow -q<x\]
		\[-x>q\Leftrightarrow -q>x\]
		where $q\in\mathbb{Q}$. Note that negation reverses orientation: if $x$ is a lower real, then its negation yields an upper real, and vice versa. Nonetheless, if $x$ is a Dedekind real (which comprises both the left and right Dedekind sections), then its negation yields another Dedekind real. As such, given two Dedekind reals $x,y$, we define their subtraction $x-y$ as $x+(-y)$.
	\end{defi}

	We next define multiplication and inverses for non-negative reals. 

	\begin{defi}\label{def:multiplication} Non-negative reals $x,y$ can be multiplied by the rules:
		\[q<x\cdot y \Leftrightarrow \exists s <x \land \exists t<y.( q<s\cdot t ) \]
		\[q>x\cdot y \Leftrightarrow \exists s >x \land \exists t>y.( q>s\cdot t ). \]
		Multiplication of arbitrary reals (not necessarily non-negative) is more involved, and will not be used in this paper. The interested reader is directed to~\cite{R} for more details.
	\end{defi}

	\begin{defi}\label{def:inverses} The inverse of a non-negative one-sided real $x$ is defined as:
		\[q<x^{-1} \Leftrightarrow x<q^{-1}\]
		\[x^{-1}<r \Leftrightarrow r^{-1}<x.\]
		Just as in the case of subtraction, inverting reverses orientation, sending lowers to uppers and vice versa.
		The definition gives that $\infty$ and 0 are inverses.
		One easily checks that (---)$^{-1}$ is an isomorphism, with  ((---)$^{-1})^{-1}$ the identity.
		They combine to give inverses of positive Dedekinds, with $x^{-1}$ the unique positive Dedekind real such that $x\cdot x^{-1}=1$.
	\end{defi}

	\begin{rem} It is an easy exercise to verify that the additive and multiplicative operations defined above turn $\ropen{0, \infty}, \overrightarrow{[0,\infty]}$ and $\overleftarrow{[0,\infty]}$ into semirings, $\mathbb{R}$ into a field, and $(0,\infty)$ a group. As such, they satisfy all the expected arithmetic identities and inequalities --- e.g.\ $x\cdot (y+z) = x\cdot y + x\cdot z$. Further, it is also straightforward to verify that multiplication preserves strict order on positive Dedekind reals (i.e.\ $x<y\implies x\cdot z < y\cdot z$) and non-strict order on the one-sided reals (i.e.\ $x\sqsubseteq y\implies x\cdot z \sqsubseteq y\cdot z$). For more details, see~\cite{R}.
	\end{rem}

	\subsection{Gluing}\label{subsec:gluingtechnique}
	Later (Theorem~\ref{thm:xzetaDed}) we would like to perform a construction for generic $x\in(0,\infty)$, except the construction must be done along a case-splitting for $x\leq 1$ and $x\geq 1$, with agreement at $x=1$.
	In point-set topology, this immediately gives a \emph{function}, after which a continuity proof is needed.
	Working point-free, however, means we must take a little care to maintain geometricity because, reasoning in sheaves over $(0,\infty)$, the conditions $x\leq 1$ and $1\leq x$ are not geometric formulae\footnote{Why are they not geometric? Answer: Because $x\leq 1$ and $1\leq x$ do not give open subspaces of $(0,\infty)$.}.

	Categorically, what we need to prove is that the left-hand diagram below is a pushout square:
	\begin{equation}%
	\label{eq:gluing}
	\begin{tikzcd}
	\{1\}
	\arrow[hook]{r}
	\arrow[hook]{d}
	& {\lopen{0,1}}
	\arrow[hook]{d}
	\\
	{\ropen{1,\infty}}
	\arrow[hook]{r}
	& (0,\infty)
	\end{tikzcd}
	\quad
	\{1\}\xbigtoto[1_L]{1_R}  \lopen{0,1}\amalg \ropen{1,\infty}\xrightarrow{P} (0,\infty)
	\end{equation}
	Equivalently, the right-hand diagram is a coequaliser,
	where $1_L$ and $1_R$ are the global points $1$ in the left and right components,
	and $P$ is the copairing of the two natural inclusions.

	We justify this by applying Vermeulen's work~\cite{Verm} on proper maps.
	There are various equivalent characterizations of ``proper'', showing the connection with the point-set notion, and the most relevant for our purposes is that a map $f\colon Y\rightarrow X$ is proper iff it is fibrewise compact [one proves geometrically that $f^{-1}(x)$ is compact for a generic $x$, and this shows that the corresponding internal space in the topos of sheaves $\baseS X$ is compact]. In particular, the class of proper surjections (cf. Definition~\ref{def:Rstructure}) possesses many nice categorical properties, such as:

	\begin{propC}[{\cite[Propositions 4.2 and 5.4]{Verm}}]\label{prop:propersurj}\leavevmode
		\begin{enumerate}
			\item Proper surjections are coequalisers (of their kernel pair).
			\item Proper surjections are stable under pullback, i.e.\ in a pullback square of spaces

			\[\begin{tikzcd}
			P \ar[r,"k"] \ar[d,swap,"h"] & Y \ar[d,"f"]\\
			Z \ar[r,"g"] & X
			\end{tikzcd}
            \]
			if $f$ is a proper surjection, then so is $h$.
		\end{enumerate}
	\end{propC}

    \noindent
	We now state and prove our gluing principle for $\lopen{0,1}$ and $\ropen{1,\infty}$
	[but the same principle holds for any interval of $\mathbb{R}$ divided at a point].

	\newcommand{\cond}[3]{(#1?\,#2\mathord{:}\,#3)}
	\begin{prop}[Gluing Principle]\label{prop:gluingcoequaliser}
		The right-hand diagram of~\eqref{eq:gluing} is a coequaliser, stable under pullbacks.
	\end{prop}
	\begin{proof}
		Our main step is to show that $P$ is a proper surjection, and hence the stable coequaliser of its kernel pair.
		After that it remains to show that it is the stable coequaliser of the pair in the statement.

		In fact we prove a stronger property of $P$, that it is an \emph{entire} surjection.
		``Entire'' means fibrewise Stone: in other words, each fibre $P^{-1}(x)$ is the spectrum (i.e.\ the space of prime filters) of a Boolean algebra $B_x$ (of clopens).
		We define $B_x$ to be the Boolean algebra presented by one generator $\alpha$ subject to the following relations:
		\[
		B_x=BA \left\langle \alpha\, \left|
		\begin{array}{l}
		\alpha= 1 \quad (\text{if}\, x<1)\\
		\alpha= 0 \quad (\text{if}\, x>1)
		\end{array} \right.
		\right\rangle%
		\text{.}
		\]
		We show that the fibre of $P$ over each $x$ is isomorphic to $\mathsf{Spec}B_x$.

		For a geometric description of points of the coproduct, we write them as $\cond{p}{y}{z}$, an abbreviation of the notation
		$\text{\textbf{if }}p\text{\textbf{ then }}y\text{\textbf{ else }}z$ used in~\cite[2.2.6]{Vi9}.
		Here, $p$ is a decidable proposition, $y$ is a point of $\lopen{0,1}$ defined if $p$ holds, and $z$ a point of $\ropen{1,\infty}$ defined if $\neg p$.
		Then, as described in~\cite{Vi9}, a copairing map $[f,g]$ maps $\cond{p}{y}{z}$ to the directed join  $\dirsup\left(\{f(y)\mid p\}\cup\{g(z)\mid\neg p\}\right)$.

		For $P$, it follows that a point of the fibre $P^{-1}(x)$ must be of the form $\cond{p}{x}{x}$.
		Thus if $x<1$ then $p$ must be true, since $x$ is not defined as point of $\ropen{1,\infty}$;
		and, similarly, if $1<x$ then $p$ must be false:
		\[
		P^{-1}(x)=
		\begin{cases}
		2, \quad\text{if $x = 1$}\\
		1, \quad\text{if $x<1$ or $1<x$}
		\end{cases}
		\]

		To map $P^{-1}(x)$ to $\mathsf{Spec}B_x$,
		we map $\cond{p}{x}{x}\mapsto
		\{1\}\cup\{\alpha\mid p\}\cup\{\neg\alpha\mid\neg p\}$.
		To show that this subset $F$ of $B_x$ is a prime filter, the main issue is that it does not contain $0$.
		To see this, consider if $\alpha=0$.
		Then $x>1$, hence $\neg p$, and $\alpha\neq F$.
		The case $\neg\alpha=0$ is similar.

		For the reverse direction we map $F\mapsto \cond{\alpha\in F}{x}{x}$.
		Note that $\alpha\in F$ is decidable; its complement is $\neg\alpha\in F$.
		If $\alpha\in F$ then $\alpha\neq 0$, so $x\leq 1$ and $x$ is defined as point of $\lopen{0,1}$.
		Similarly, if $\neg\alpha\in F$ then $x$ is defined as point of $\ropen{1,\infty}$.
		It follows that $\cond{\alpha\in F}{x}{x}$ is a point of $P^{-1}(x)$.

		The two maps are mutually inverse, which proves our claim that $P$ is entire. It is surjective because every $B_x$ is non-degenerate (i.e.\ it has $1\neq 0$), essentially because we cannot have both $x < 1$ and $x>1$.

		Now we know that $P$ is the coequaliser of its kernel pair, it remains to show that the kernel pair and the pair $(1_R,1_L)$ have the same coequalisers.
		The kernel pair, the pullback of $P$ against itself, can be calculated as the coproduct of four pairwise pullbacks of the components of
		$\lopen{0,1}\amalg\ropen{1,\infty}$.
		Since both $\lopen{0,1}$ and $\ropen{1,\infty}$ are embedded in $(0,\infty)$, their kernel pairs are just the reflexive parts and are irrelevant to the coequaliser.
		The pullback $\lopen{0,1} \times_{(0,\infty)}\ropen{1,\infty}$,
		the space of pairs $(x,x)$ such that $1\geq x \geq 1$, is just $\{1\}$,
		and the remaining component is just the reverse of that, and implied by symmetry.
		Hence the kernel pair has the same coequaliser as the pair $(1_L,1_R)$.
	\end{proof}

	\begin{cor}\label{cor:pushoutproperty}
		The left-hand diagram in~\eqref{eq:gluing} is a pushout square.
	\end{cor}

	\begin{discussion}[Stability under pullback]
    A reasonable question: why do we ask for the coequaliser to be stable in the Gluing Principle? The short answer: geometricity. To elaborate, saying that a construction is geometric is equivalent to saying that it's stable under pullback along geometric morphisms. The Gluing Principle is meant to provide a geometric justification for the case-splitting along $x\leq 1$ and $x\geq 1$ --- the pushout property gives us a framework for gluing the two cases together, while its stability under pullback tells us that the gluing is geometric.
	\end{discussion}

	\section{Rational Exponents}\label{sec:ExpI} In this section, we develop the consequences of two pairs of basic \emph{exponent laws:}
	\begin{gather}
	x^{\zeta+\zeta'}=x^\zeta x^{\zeta'}\text{, }
	\quad x^0=1%
	\label{eq:BE1}
	\\
	x^{\zeta\cdot \zeta'}=(x^\zeta)^{\zeta'}\text{, }
	\quad x^1 = x%
	\label{eq:BE2}
	\end{gather}
	In describing a map $(x, \zeta) \mapsto x^\zeta$ as an \emph{exponentiation}, we shall mean that it satisfies the above exponent laws. They are enough to prescribe what $x^\zeta$ has to be for $\zeta$ rational. With $a$ a natural number, $x^a$
	must be by repeated multiplication; for $b$ a positive natural number, $x^{\frac{1}{b}}$ must be a radical, and $x^{\frac{a}{b}}$ combines those; and $x^{-\zeta}$ is $(x^\zeta)^{-1}$.
	For completeness, we shall also prove the following \emph{base product} law:
	\begin{gather}
	(xy)^\zeta = x^\zeta y^\zeta\text{, }
	\quad 1^\zeta = 1%
	\label{eq:BE3}
	\end{gather}
	These identities recover the familiar (and standard) algebraic properties of exponentiation. Hereafter, we shall refer to Equations~\eqref{eq:BE1}--\eqref{eq:BE3} collectively as \emph{the Basic Equations}.

	Before proceeding, however, first some obligatory remarks about the appropriate range for the base $x$. Clearly, without complex numbers we cannot hope to deal with radicals of negative reals, so we shall have to assume $x\geq 0$. Further, in later sections, we shall also find two additional problems with the case $x=0$. The first (Section~\ref{subsec:signedratexp}) is for negative exponents and Dedekind base $x$ since $x$ will need to be invertible. The second (Section~\ref{sec:ExpII}) is that while our definition $x^0=1$ is OK for rational exponents (cf. Discussion~\ref{disc:undecidableorder}), this causes continuity issues for real exponents; indeed, this reflects the classical fact that $0^0$ is not well-defined. Nonetheless, so long as we work with just non-negative rational exponents, the assumption that $x\geq 0$ is OK\@.

	\subsection{Natural Number Exponentiation of Discrete Monoids}%
	\label{subsec:NatNumDisc} Let $M$ be a set equipped with a multiplicative monoid structure. Let $x\in M$ be an element of such a set-based multiplicative monoid. By the universal property of $\mathbb{N}$ being the free monoid generated by 1, we obtain a unique monoid homomorphism corresponding to the set-based function sending $1$ to $x$ in $M$. This yields the following map:
	\begin{align*}
	M\times\mathbb{N}&\rightarrow M\\
	(x,a)&\mapsto x^a
	\end{align*}

	\begin{prop}\label{prop:BasicEqns}
		If $M$ is commutative, then exponentiation (as defined above) satisfies the Basic Equations.
	\end{prop}
	\begin{proof} Let $x\in M$, and $a,a'\in\mathbb{N}$.
		Equation~\eqref{eq:BE1}, and the second part of Equation~\eqref{eq:BE2}, come straight from the definition.
		The others are by induction on $a$ or $a'$.
	\end{proof}

	Obviously it will be our aim to show these equations for real exponentiation. As our first step towards this goal, recall that $Q$ denotes the set of non-negative rationals. It is known geometrically that $Q$ is a monoid with respect to multiplication, hence we obtain the following exponentiation map as a special case of the previous construction:
	\begin{align*} Q\times\mathbb{N}&\rightarrow Q\\
	(x,a)&\mapsto x^a
	\end{align*}

	We finish this subsection by establishing some important (and useful!) algebraic properties of this exponentiation map:

	\begin{lem}[Monotonicity]\label{lem:monotonerationals} Denote $\mathbb{N}_+$ to be the set of positive natural numbers.
		If $a\in \mathbb{N}_+$, then the map $\argu^a$ on non-negative rational base preserves and reflects the strict order, and is also unbounded.
		\begin{proof} We prove that  $\argu^a$ preserves the strict order (i.e.\ is strictly monotonic) on $Q$  by induction on $a$. The base case is trivial since $x<y\implies x^1<y^1$ for any $x,y\in Q$ by the basic exponent laws. For the inductive step, suppose that $x<y$ implies that $x^a<y^a$.  Then this yields: $x^{a+1}=x^a\cdot x<x^a\cdot y< y^a\cdot y = y^{a+1}$. To prove that $\argu^a$ reflects the strict order, we have to show that $x^a<y^a\implies x<y$. Decidability of $<$ on $Q$ means that either $x<y$ or $x\geq y$.  Since monotonicity means $x\geq y\implies x^a\geq y^a$, contradicting our hypothesis that $x^a<y^a$, this means that the remaining case $x<y$ must be true.
			Finally, we show that $\argu^a$ is unbounded. Given any $x\in Q$ such that $1<x$, note that monotonicity implies that $1<x^a$. A simple inductive argument easily shows that $x<x^a$. \end{proof}\end{lem}

	\begin{lem}\label{lem:rational sandwich} For any pair of non-negative rationals $q,r\in Q$ such that $q<r$, and any positive natural number $a\in\mathbb{N}_+$, there exists a positive rational $s\in Q_+$ so that $q<s^a<r$.
		\begin{proof} Given $q,r\in Q$ such that $q<r$, denote $\epsilon:=\frac{r-q}{2}$. By Lemma~\ref{lem:monotonerationals}, we know that $\argu^a$ is unbounded on $Q$ --- thus there exists some $v\in Q$ such that $v^a>r$. Consider the function $\argu^a$ on the set of rationals from $0$ to $v$. For any two rationals $x,y$ such that $0\leq x<y\leq v$, we have:
			\begin{align*}
			y^a-x^a&=(y-x)\cdot (y^{a-1}+y^{a-2}\cdot x + \dots + y\cdot x^{a-2}+x^{a-1})\\
			&\leq (y-x)\cdot (v^{a-1}+v^{a-2}\cdot v + \dots  + v\cdot v^{a-2}+v^{a-1})\\
			& = (y-x)\cdot a\cdot v^{a-1}
			\text{.}
			\end{align*}

			\noindent By the Archimedean property (Fact~\ref{fact:archproperty}), there exists some $M\in\mathbb{N}$ such that $\frac{1}{M}<\frac{\epsilon}{a\cdot v^a}$. Denoting $s_i:=\frac{i\cdot v}{M}$, it is clear that:
			\[\displaystyle\bigcup_{0\leq i\leq M}[s_i,s_{i+1}]=[0,v]\]
			and that $s_{i+1}-s_i = \frac{v\cdot (i+1-i)}{M}=\frac{v}{M}<\frac{\epsilon}{a\cdot v^{a-1}}$, and so $s_{i+1}^a-s_i^a < \epsilon$.

			Next, one easily verifies that since:
			\begin{itemize}
				\item 		 $s_0^a=0\leq q$;
				\item $s_M^a=v^a>r>q$; and
				\item  $\{s_i^a\}$ is strictly monotone in $i\in \{0,1.,...,M\}$;
			\end{itemize}
			there thus exists a unique $j$ such that $s_j^a\leq q<s_{j+1}^a$.
			Recalling that $\epsilon:=\frac{r-q}{2}$, this consequently yields:
			\[q<s_{j+1}^a\leq s_{j}^a + \epsilon \leq q +\epsilon <r\]
			proving the lemma, i.e.\  $\exists s\in Q_+$ such that $q<s^a<r$.\end{proof}
	\end{lem}

	\subsection{Natural Number Exponentiation of Non-negative Reals}%
	\label{subsec:NatNumReals}
	Unfortunately, the previous exponentiation map is geometric only for monoid structures on sets, as opposed to topological monoids. Hence, in order to generalise $x^a$ to the case where $x$ is real, adjustments will be needed. The good news is there is a way in which we can lift the exponentiation map from the rational case to the real case, which we make precise in the following proposition:

	\begin{prop}\label{prop:explowerupperreals} The map $\argu^a$ that sends $(x,a)\mapsto x^a$ on rationals extends to two maps on one-sided reals:
		\begin{align*}
		\overrightarrow{[0,\infty]}\times\mathbb{N}&\longrightarrow  \overrightarrow{[0,\infty]}\\
		\overleftarrow{[0,\infty]}\times\mathbb{N}&\longrightarrow  \overleftarrow{[0,\infty]}
		\text{.}
		\end{align*}
		Each map is unique subject to being monoid homomorphism for $n$ and with $x^1=x$. In fact, all the Basic Equations~\eqref{eq:BE1}--\eqref{eq:BE3} hold.
	\end{prop}
	\begin{proof} Let us fix some $a\in\mathbb{N}$. Then, consider the diagram:
		\[\begin{tikzcd} Q  \ar[r, "f_a"] \ar[d,"\phi"] & Q \ar[d,"\phi"]\\
		\overrightarrow{[0,\infty]} \ar[r, dashed, "\overline{f_a}"] & \overrightarrow{[0,\infty]}
		\end{tikzcd}
        \]
		where $f_a\colon Q\rightarrow Q$ is the map that sends $x\mapsto x^a$. We now check that $\phi\circ f_a$ satisfies the two lifting conditions of the Lifting Lemma based on $\overrightarrow{[0,\infty]}\cong\RIdl(Q,<)$, where $<$ is the usual strict order on $Q$ but modified to allow $0<0$ (cf. Remark~\ref{rem:liftingvariation}).

		For monotonicity, suppose $q<q'$ for $q,q'\in Q$. By Lemma~\ref{lem:monotonerationals}, this yields the inequality $q^a\leq (q')^a$, which is preserved by $\phi$, and so $q<q'\!\implies\! \phi\circ f_a(q)\sqsubseteq \phi\circ f_a(q')$. To verify continuity, this amounts to showing that if $r<(q')^a$, then $\exists q<q'$ such that $r<q^a$ (where $r,q,q'\in Q$). For $a=0$, $r<(q')^{0}=1$, then we can let $q=\frac{q'}{2}$ since $r<1=q^{0}$. For $a\geq 1$, we know by Lemma~\ref{lem:rational sandwich} that there exists some $t\in Q$ such that $r<t^a<(q')^a$. Since $\argu^a$ still reflects the (modified) order $<$ on $Q$, we have that $t<q'$ and $r<t^a$, as desired.  With the appropriate hypotheses verified, we apply the Lifting Lemma to obtain the (unique) map $\overline{f}\colon \overrightarrow{[0,\infty]}\rightarrow \overrightarrow{[0,\infty]}$. Viewing this map externally (cf. Convention~\ref{conv:fixingx}), we get an exponentiation map $\overrightarrow{[0,\infty]}\times\mathbb{N}\rightarrow \overrightarrow{[0,\infty]}$ sending $(x,a)\mapsto x^a$.
		The Basic Equations, and also the uniqueness, follow from surjectivity of $\phi$, since we already know that they hold for rational $x$ and $y$.

		A similar argument (using $Q_{+}\cup\{\infty\}$) works for the upper reals, thus defining $x^a$ for some upper real $x\in \overleftarrow{[0,\infty]}$.
	\end{proof}

	More explicitly, Proposition~\ref{prop:explowerupperreals} extends natural number exponentiation $\argu^a$ from the rationals to the one-sided reals, yielding:
	\begin{align*}
	q < x^a &\Leftrightarrow
	q<0\lor \exists q'\in Q.(q'<x\land q<(q')^a) \\
	x^a < r & \Leftrightarrow
	\exists r'\in Q_+.(x<r'\land (r')^a<r).
	\end{align*}
	on $\overrightarrow{[0,\infty]}$ and $\overleftarrow{[0,\infty]}$ respectively. Putting everything together, we now define natural number exponentiation on the Dedekind reals as follows:

	\begin{prop}\label{prop:NatNumReals} Natural number exponentiation on the one-sided reals (as per Proposition~\ref{prop:explowerupperreals}) combine to yield the following map on the Dedekind reals:
		\begin{align*}
		\ropen{0, \infty}\times \mathbb{N}&\rightarrow \ropen{0, \infty}\\
		(x,a)&\mapsto x^{a} = (L_x^a, R_x^a)\text{.}
		\end{align*}
		The Basic Equations~\eqref{eq:BE1}--\eqref{eq:BE3} hold for Dedekind $x$ and $y$.
	\end{prop}
	\begin{proof} It remains to verify the inhabitedness, separatedness and locatedness axioms. Right inhabitedness essentially follows from the unboundedness of $\argu^a$ on $\mathbb{Q}$ (Lemma~\ref{lem:monotonerationals}) whereas we get left inhabitedness (and non-negativity) for free since for any negative rational $q<0$, we get $q<x^{a}$ by construction.

		For separatedness, suppose $q < x^a < r$. Notice that $r\in Q_+$ by construction, hence if $q\leq 0$, then $q<r$ automatically. Suppose instead that $q>0$ and there exist non-negative rationals $s,t$ such that $s<x<t$, whereby $q< s^a$ and $t^a<r$. Since $x$ is a \textit{separated} Dedekind, $s<x<t\implies s<t$. Hence, since $s^a=1=t^a$ if $a=0$, and $s^a<t^a$ if $a>0$ by Lemma~\ref{lem:monotonerationals}, this combines to yield the inequality $q<s^a\leq t^a<r$, proving separatedness.

		For locatedness, suppose we have $q,r\in\mathbb{Q}$ such that $q<r$. If $q<0$, then $q<x^{a}$ automatically, so assume $q\geq 0$ for the remainder of this proof.

		In the case where $a>0$, by Lemma~\ref{lem:rational sandwich} we can find $q',r'\in Q_+$ such that $q<(q')^a<(r')^a<r$. Since $\argu^a$ reflects strict order (Lemma~\ref{lem:monotonerationals}), we get $(q')^a<(r')^a\implies q'<r'$. Further, since $x$ is a \textit{located} Dedekind real, this implies that either $q'<x$ (and thus $q<x^a$) or that $r'>x$ (and thus $r>x^a$). Alternatively, suppose that $a=0$. Then locatedness is obvious, since $x^0=1$ is located. Hence, in either case ($a=0$ or $a>0$), locatedness holds.

		The Basic Equations follow immediately, because we know they hold for the lower and upper parts.
	\end{proof}

	\begin{conv}[``Non-negative Reals'']
    Whenever we state that a result holds for the ``non-negative reals'', we shall mean that it holds for both the non-negative one-sideds (i.e.\ $x\in\overrightarrow{[0,\infty]}$ or $x\in\overleftarrow{[0,\infty]}$) and the non-negative Dedekinds (i.e.\ $x\in\ropen{0, \infty}$) --- which are the cases for which we have defined exponentiation in Propositions~\ref{prop:explowerupperreals} and~\ref{prop:NatNumReals}. Whenever we wish to prove a sharper result that holds just in the case of non-negative Dedekinds, we shall signpost this explicitly.
	\end{conv}

	We end this section by generalising the monotonicity principle of Lemma~\ref{lem:monotonerationals}:

	\begin{lem}[Monotonicity]\label{lem:monotonereals} Let  $a\in\mathbb{N}_+$ be a positive natural number. Then the map $\argu^a$ preserves and reflects non-strict order on the non-negative reals. Further, it is also an unbounded map that preserves and reflects strict order on the non-negative Dedekind reals.
	\end{lem}
	\begin{proof} Consider $\argu^a$ on the lower reals $\overrightarrow{[0,\infty]}$. Preservation of non-strict order follows immediately from the fact that any continuous map preserves specialisation order.
		To show that $\argu^a$ reflects non-strict order, suppose that $x^a\sqsubseteq y^a$. This yields the computation:
		\begin{align*}
		q<x &\implies q^a < x^a\quad\quad \quad \quad\quad \quad\quad \quad \quad\quad  \quad\!\!\!\!\!\text{[by unwinding the definition of $\argu^a$]}\\
		& \implies q^a < y^a \quad\quad \quad \quad\quad \quad\quad \quad \quad\quad  \quad\!\!\!\!\!\text{[since $x^a\sqsubseteq y^a$] }\\
		& \implies \exists q'\in Q.(q'<y\land q^a<(q')^a) \,\quad\text{[by construction/definition of $\argu^a$]}\\
		&\implies q<q' \, \quad\quad \quad \quad\quad \quad\quad \quad \quad\quad  \quad\!\text{[since $\argu^a$ reflects strict order on $Q$]}\\
		&\implies q<y  \quad\quad \quad \quad\quad \quad\quad \quad \quad\quad  \quad \!\!\!\!
		\!\quad\text{[by downward closure of lower reals]}
		\end{align*}

		Next, to show that $\argu^a$ also preserves strict order on the non-negative Dedekinds, suppose $x<y$, and so there exists rationals $q,q'$ such that $x<q<q'<y$. Then since $\argu^a$ preserves non-strict order, and using Lemma~\ref{lem:monotonerationals}, we get $x^a\leq q^a< (q')^{a} \leq y^a$ and hence $x^a<y^a$.

		On the other hand, suppose $x^a<y^a$. Since $0\leq x^a$, we may apply Lemma~\ref{lem:rational sandwich}, to obtain positive rationals $s,s'\in Q_+$ such that $x^a<s^a<(s')^a<y^a$. Since $s<s'$ iff $s^a<(s')^a$, and since $\argu^a$ reflects non-strict order, we get  $x\leq s < s' \leq y$, which shows that $\argu^a$ also reflects strict order on the non-negative Dedekinds.

		Unboundedness follows directly from the rational case, Lemma~\ref{lem:monotonerationals}.
	\end{proof}

	\subsection{Radicals of Non-Negative Reals}%
	\label{subsec:RootsReals} Next, given some $b\in\mathbb{N}_+$, we would like to define the $b$th-root of a non-negative real. Unlike the previous subsection, we shall define this directly (as opposed to first working with the rationals before lifting to the reals):

	\begin{prop}\label{prop:RootsReals}
		Define maps $(x,b)\mapsto x^{\frac{1}{b}}$ on the non-negative reals using:
		\[q<x^{\frac{1}{b}}\Leftrightarrow  q<0\lor \left(q^b<x\land q\in Q\right)\]
		\[r>x^{\frac{1}{b}}\Leftrightarrow r^b>x.\]
		Then $x^{\frac{1}{b}}$ is a real of the same kind as $x$ (lower, upper, or Dedekind).
	\end{prop}

	\begin{proof} Let $x$ be a non-negative lower real. Non-negativity of $x^{\frac{1}{b}}$ is immediate from definition. When $q<0$, downward closure and roundedness are obvious, so we shall assume that $q^b<x\land q\in Q$. In which case, downward closure says that if $q'<q$ for $q'\in \mathbb{Q}$, then $q<x^\frac{1}{b}\implies q'<x^{\frac{1}{b}}$. If $q'<0$, this is obvious. If $q'\geq 0$, this follows immediately from the monotonicity of exponentiation by $b\in\mathbb{N}_+$ (Lemma~\ref{lem:monotonerationals}). As for roundedness, we must show if $q^b<x$ then there exists rational $q'>q$ such that $q^b<(q')^b<x$. We know there exists $r\in Q_+$ such that $q^b<r<x$. By Lemma~\ref{lem:rational sandwich}, there exists $q'\in Q_+$ such that $q^b<(q')^b<r<x$, and we are done.

		The case for non-negative upper reals is analogous.

		As before, to define $\argu^{\frac{1}{b}}$ on non-negative Dedekinds, we shall need to verify the rest of the axioms from Definition~\ref{def:DedekindReals}. Left inhabitedness comes for free since $q<x^{\frac{1}{b}}$ for all negative rationals $q$. Right inhabitedness follows from the unboundedness of $\argu^b$ on the rationals (Lemma~\ref{lem:monotonerationals}).

		For separatedness, suppose $q<x^{\frac{1}{b}}<r$. Since $x$ is non-negative, and since $\argu^b$ reflects strict order, this immediately implies that $r>0$. As such, if $q<0$, then $q<r$ automatically. Hence, suppose instead that $q\geq 0$, and that $q^b<x<r^b$. Since $x$ is a \textit{separated} Dedekind, this implies that $q^b<r^b$, and so we get $q<r$ (again by Lemma~\ref{lem:monotonerationals}).

		Finally, we check locatedness. Suppose $q<r$. Again, if $q<0$, then we get that $q<x^{\frac{1}{b}}$ for free, so suppose $q\geq 0$. By Lemma~\ref{lem:monotonerationals}, we know that $0\leq  q^b < r^b$. Since $x$ is a \textit{located} Dedekind, this implies that either $q^b<x\vee x<r^b$, which in turn implies (by construction) that $q<x^{\frac{1}{b}}\vee x^{\frac{1}{b}}<r$, proving the axiom.
	\end{proof}

	The key property of $b$th roots, of course is that taking the root is inverse to raising to the power.

	\begin{prop}%
		\label{prop:rootsinteractexp} Given a non-negative real $x$, and $0\neq b\in\mathbb{N}$, we have that
		\[
		x=(x^{\frac{1}{b}})^b=(x^b)^{\frac{1}{b}}
		\text{.}
		\]
	\end{prop}
	\begin{proof} The proof of these identities for the one-sided reals is analogous, so we shall only prove it for the upper reals, which will also automatically extend the result to the Dedekind reals (cf. Corollary~\ref{cor:check1sided}).

		To prove  $x=(x^{\frac{1}{b}})^b$ for the upper reals, suppose $q>(x^\frac{1}{b})^b$. This means that $\exists q'\in Q_+$ such that $(q')^b>x$ and $q>(q')^b$, which implies that $q>x$. Conversely, suppose $q>x$. By roundedness, we know that there exists some $q''\in Q_+$ such that $q>q''>x\geq 0$. By Lemma~\ref{lem:rational sandwich}, we know that there exists some $q'\in Q_+$ such that $q>(q')^b>q''>x$, proving that $q>(x^{\frac{1}{b}})^b$.

		To prove $x=(x^b)^{\frac{1}{b}}$, note that $x=(x^\frac{1}{b})^{b}$ implies that $\big((x^b)^{\frac{1}{b}}\big)^b=x^b$. The result then follows since $\argu^b$ reflects the (non-strict) order on the upper reals by Lemma~\ref{lem:monotonereals}.
	\end{proof}

	\begin{cor}%
		\label{cor:rootsinteractexp}
		The following equations hold for $x,y$ non-negative reals, and $a,b,d\in\mathbb{N}$ where $b,d\neq 0$:
		\begin{gather*}
		(x^a)^{\frac1b}=(x^{\frac1b})^a
		\\
		x^{\frac1{bd}}=(x^{\frac1b})^{\frac1d}\text{,}\quad x^{\frac11}=x
		\\
		(xy)^{\frac1b}=x^{\frac1b}y^{\frac1b}.
		\end{gather*}
	\end{cor}
	\begin{proof}

		In each case, the proof is to raise both sides to an appropriate power,
		use equations already known for integer powers, and then take the root.
		For example, the first one follows from
		\[
		((x^{\frac1b})^a)^b
		= ((x^{\frac1b})^{ab}
		= ((x^{\frac1b})^b)^a = x^a
		\text{.}
        \qedhere
		\]
	\end{proof}

	\subsection{Non-negative Rational Exponents}%
	\label{subsec:RationalExpReals}
	Having defined $x^a$ and $x^{\frac{1}{b}}$ for $a\in\mathbb{N}$ and $b\in\mathbb{N}_+$ we combine these two constructions together as
	\[x^{\frac{a}{b}}=\big(x^a\big)^{\frac{1}{b}}\]
	Note that we get the fact that $x^{\frac{a}{b}}$ is a non-negative real for free due to Propositions~\ref{prop:NatNumReals} and~\ref{prop:RootsReals}. The only thing left to check is that this construction is well-defined with respect to the equivalence of rationals.

	\begin{prop}\label{prop:RationalExpReals} The exponential $x^q$,
		with $x$ being a non-negative real and $q$ a non-negative rational, is well-defined and
		satisfies the Basic Equations~\eqref{eq:BE1}--\eqref{eq:BE3}.
	\end{prop}
	\begin{proof} To show that $x^q$ is well-defined (with respect to the equivalence of rationals), we need to show that given any $\frac{a}{b}=\frac{c}{d}$, where $a,c\in\mathbb{N}$ and $b,d\in\mathbb{N}_+$, we have that $(x^{a})^{\frac{1}{b}}=(x^{c})^{\frac{1}{d}}$.
		Why is this? First, note that $x^{\frac{a}{b}}=x^{\frac{ak}{bk}}$, for $\frac{a}{b}\in Q$ and $k\in\mathbb{N}_+$. Indeed, by Proposition~\ref{prop:rootsinteractexp} and Corollary~\ref{cor:rootsinteractexp}, we have:
		\[x^{\frac{ak}{bk}}=(x^{ak})^{\frac{1}{bk}}=\bigg(\big((x^a)^k\big)^{\frac{1}{k}}\bigg)^\frac{1}{b}=x^{\frac{a}{b}}.\]
		More generally, suppose we have that $\frac{a}{b},\frac{c}{d}\in Q$ such that $\frac{a}{b}=\frac{c}{d}$. This obviously implies that $ad=bc$, and thus our previous computation yields the identity: $x^{\frac{a}{b}}=x^{\frac{ad}{bd}}=x^{\frac{bc}{bd}}=x^{\frac{c}{d}}$, as desired.

		To see why the Basic Equations hold for non-negative rational exponents, this follows follows algebraically from the Basic Equations already established and Corollary~\ref{cor:rootsinteractexp}.
		For example for the law of adding exponents, if $r=\frac{a}{b}$
		and $s=\frac{c}{d}$, then
		\[
		x^{r+s}=x^{\frac{ad+bc}{bd}} = \bigg(x^{\frac{1}{bd}} \bigg)^{ad+bc} = \bigg (x^{\frac{1}{bd}}\bigg)^{ad}\cdot \bigg (x^{\frac{1}{bd}}\bigg)^{bc} = x^{\frac{a}{b}}\cdot x^{\frac{c}{d}}=x^r\cdot x^s.
        \qedhere
		\]
	\end{proof}

	\begin{lem}\label{lem:monotonepositiverationals} Fix $q\in Q_+$. Then $\argu^q$ preserves and reflects non-strict order on non-negative reals. Further, it is an unbounded map that preserves and reflects strict order on non-negative Dedekind reals.
	\end{lem}
	\begin{proof} Preservation and reflection of strict (resp.\ non-strict) order on the positive Dedekind (resp.\ one-sided) reals is immediate from Lemma~\ref{lem:monotonereals}. Further, express $q$ as $\frac{a}{b}$ for $a,b\in\mathbb{N}_+$. We know for any positive Dedekind $x\in(0,\infty)$ that there exists $s\in\mathbb{Q}$ whereby $1<s$ and $x<s\leq s^a$. Since $s^a=(s^b)^\frac{a}{b}=(s^b)^q$, this proves that $\argu^q$ is unbounded.
	\end{proof}

	\subsection{Signed Rational Exponents}\label{subsec:signedratexp}  In this subsection, we extend the previous definition of rational exponentiation to also include the non-positive rationals. Here we must restrict to the case where the base is Dedekind but not one-sided. This is because inverting reverses orientation --- applying Definition~\ref{def:inverses}, the inverse of a lower real $x$ yields an upper real (and vice versa). Hence, much like subtraction, whilst negative exponents are well-defined on the Dedekinds, they are not well-defined on just the lower or upper reals alone. We further require the base to be \textit{positive} Dedekind as well since inverses are only well-defined for non-zero Dedekinds.

	Recall from Definition~\ref{def:inverses} that given any positive Dedekind $x\in (0,\infty)$, there exists a unique inverse $x^{-1}\in (0,\infty)$ such that $x^{-1}\cdot x=1$.

	\begin{defi}\label{def:signedratexp} Let $x\in(0,\infty)$, and $q\in\mathbb{Q}$. We define:
		\[x^q=\begin{cases}
		\, x^q  \quad\quad\qquad\,\,\,\,\,\,\text{if}\, q\geq 0\\
		\, (x^{-q})^{-1}\,\,\,\,\qquad\text{if}\, q\leq  0
		\end{cases}\]
	\end{defi}
	Using Lemma~\ref{lem:monotonepositiverationals} we can see that $0<x^q$. Further, we remark that this definition of non-positive exponentiation justifies our notation of denoting inverses as $\argu^{-1}$, as can be seen from the following lemma:

	\begin{lem}\label{lem:negativeinverseidentity} Fix $x\in(0,\infty)$. For any $q\in \mathbb{Q}$, we have that $x^{-q}=(x^{q})^{-1}=(x^{-1})^{q}$.
	\end{lem}
	\begin{proof}
	    For the first identity, if $q\geq 0$ then $x^{-q}=\left(x^{q}\right)^{-1}$ by definition;
	    while if $q \leq 0$ then
	    $\left(x^{q}\right)^{-1}
	      =\left(\left(x^{-q}\right)^{-1}\right)^{-1}=x^{-q}$.

	    As for the identity $(x^{q})^{-1}=(x^{-1})^{q}$, if $q\geq 0$ then it follows from the Basic Equations because

		\[x^{q}\cdot\left(x^{-1}\right)^{q}=\left(x\cdot x^{-1}\right)^{q}=1\text{.}\]
		Then for $q\leq 0$ we have
		\[\left(x^{-1}\right)^q
		    = \left(\left(x^{-1}\right)^{-q}\right)^{-1}
		    = \left(\left(x^{-q}\right)^{-1}\right)^{-1}
		    = x^{-q}
		    \text{.}
            \qedhere
		\]
%
	\end{proof}

	\begin{rem}[Gluing maps defined on subspaces of $\mathbb{Q}$ vs.  $\,\!\!\mathbb{R}$]%
    \label{rem:gluingissues} There is a subtle geometricity issue hidden in our construction that bears highlighting. Definition~\ref{def:signedratexp} hinges upon a case-splitting: we gave two separate definitions of $x^q$ based on whether $q\leq 0$ or $q\geq 0$, and (implicitly) claimed that this presents a geometric account of $x^q$ for all $q\in\mathbb{Q}$. Why is this? The short answer: unlike the case for $\mathbb{R}$, we get geometricity of the case-splitting essentially for free since $<$ is decidable on $\mathbb{Q}$ (cf. Discussion~\ref{disc:undecidableorder}).
	\end{rem}

	We conclude by proving the Basic Equations~\eqref{eq:BE1}--\eqref{eq:BE3}. A common theme runs through the proofs: for signed rational exponentiation, we must now keep track of how non-negative and non-positive exponents interact with one another, forcing us to consider the various possible cases. Nonetheless, most of these can be handled similarly (modulo some technical adjustments) and so the case-splitting primarily serves as a form of bookkeeping as opposed to being a sign of some hidden complexity.

	\begin{prop}\label{prop:ratBE2} Let $x\in (0,\infty)$, and $q,q'\in\mathbb{R}$. Then $(x^{q})^{q'}=x^{q\cdot q'}$.
	\end{prop}
	\begin{proof} Strict order is decidable on $\mathbb{Q}$, and so given any $q\in\mathbb{Q}$, we can split our proof into the cases when $q,q'$ have the same signs or opposite signs.

    \smallskip
		\textbf{Case 1:} $q,q'\geq 0$. By Proposition~\ref{prop:RationalExpReals}.

		\textbf{Case 2:} $q,q'\leq 0$. This follows from Case 1 and Lemma~\ref{lem:negativeinverseidentity}, which yields:
		\[(x^{q})^{q'}=\bigg(\big((x^{-q})^{-1}\big)^{-q'}\bigg)^{-1}=\bigg(\big((x^{-q})^{-q'}\big)^{-1}\bigg)^{-1}=\bigg(\big(x^{q\cdot q'}\big)^{-1}\bigg) ^{-1}=x^{q\cdot q'}\]

		\textbf{Case 3:} $q\leq 0\leq q'$. This also follows from Case 1 and Lemma~\ref{lem:negativeinverseidentity}, since:
		\[(x^{q})^{q'}=((x^{-q})^{-1}\big)^{q'}=\big((x^{-q})^{q'}\big)^{-1}=(x^{-q\cdot q'})^{-1}=\big(x^{q\cdot q'})^{-1}\big)^{-1}=x^{q\cdot q'}\]

		\textbf{Case 4:} $q'\leq 0\leq q$. By symmetry with Case 3.
	\end{proof}

	\begin{prop}\label{prop:ratBE1} Let $x\in(0,\infty)$, and $q,q'\in\mathbb{Q}$. Then $x^{q+q'}=x^{q}\cdot x^{q'}$.
	\end{prop}

	\begin{proof} Similar to Proposition~\ref{prop:ratBE2}, we split our proof into various cases, based on the sign of the rational exponents of the identity. By previous work, we have already shown the following case:

    \medskip
		\textbf{Basic Case:} $q,q'\geq 0$ (and thus $q+q'\geq 0$): Immediate from Proposition~\ref{prop:RationalExpReals}.

    \medskip
		We claim that all the possible (signed) combinations of $q,q'$ and $q+q'$ ultimately reduce to this basic case after some elementary algebraic manipulations. If at least one of them, say $q$, is negative, then the equation to prove is equivalent to
		\[x^{q'}=x^{q+q'}\cdot x^{-q},\]
		once we multiply both sides by $x^{-q}=(x^{q})^{-1}$. This is
        another instance of the identity in the statement, but with $(q,q')$ replaced by $(-q,q+q')$, with strictly fewer of the three exponents negative. If in addition $q+q'$ and/or $q'$ are also negative, then we iterate the process so that we eventually hit the Basic Case.
	\end{proof}

	\begin{prop} Let $x,y\in(0,\infty)$, and $q\in\mathbb{Q}$. Then $(x\cdot y)^{q}=x^q\cdot y^q$.
	\end{prop}
	\begin{proof} Immediate from Proposition~\ref{prop:RationalExpReals} and definitions.
	\end{proof}

	\section{Real Exponents (The General Case)}\label{sec:ExpII}
	Moving on to real exponents, we have to be careful with monotonicity if we are to include one-sided
	reals. This is because any map must be monotone with respect to the specialisation order. Hence, if an
	argument is a one-sided real, then the result is numerically monotone with respect to that argument if it is
	the same orientation, antitone if opposite.

	Fixing $\zeta$, the map $\argu^\zeta$ is monotone or antitone in $x$ according as $\zeta\geq0$ or $\zeta\leq0$ --- this follows immediately from the fact that inverting reverses orientation. Fixing $x$, the map $x^\argu$ is monotone or antitone in $\zeta$ according as $x\geq  1$ or $x\leq 1$ --- this is clearly seen in the case of rational exponents:

	\begin{prop}[Monotone/Antitone behaviour of rational exponents]\label{prop:monotoneantitone}\leavevmode
		\begin{enumerate}[label={(\roman*)}]
			\item 	Fix Dedekind real $x$ such that $x>1$. Then the map $x^{\argu}$ is strictly increasing on $\mathbb{Q}$.
			\item Fix Dedekind real $x$ such that $0< x< 1$. Then the  map $x^{\argu}$ is strictly decreasing on $\mathbb{Q}$.
			\item Fix one-sided real $x$ such that $x\geq 1$. Then the map $x^{\argu}$ is non-strictly increasing on $Q$.
			\item Fix one-sided real $x$ such that $0\leq x\leq 1$. Then the map $x^{\argu}$ is non-strictly decreasing on $Q$.
		\end{enumerate}
	\end{prop}
	\begin{proof} Fix Dedekind $x$ such that $x>1$. Suppose we are given some $r,s\in\mathbb{Q}$ such that $r<s$. By the exponent laws for rational exponents, this implies that $x^{s}=x^{r+s-r}=x^{r}\cdot x^{s-r}$. By Lemma~\ref{lem:monotonepositiverationals}, we have that $1<x^q$ for any positive rational $q>0$. Since $0<s-r\implies 1<x^{s-r}$, this in turn implies that $x^s=x^{r}\cdot x^{s-r}>x^{r}$, proving that $x^{\argu}$ is indeed strictly increasing on $\mathbb{Q}$. The case for $x^{\argu}$ when we have Dedekind $0<x<1$ is analogous.

		The same proof works for the one-sided case modulo the following adjustments: one, $x^{\argu}$ is not defined for negative rational exponents, so we restrict the map to just the non-negative rationals $Q$; and two,  Lemma~\ref{lem:monotonepositiverationals} now holds that $\argu^q$ only preserves \emph{non-strict} order, so we only get weakly monotonic/antitonic results for $x^{\argu}$ when $x$ is one-sided.
	\end{proof}

	What are the implications of these varying monotonicity behaviours? For one-sided real arguments, this fragments the exponentiation into different cases based on the ranges of values and the one-sided orientations. We present the possibilities in the table below.
	Each table entry shows the type of $x^\zeta$ for given types of $x$ and $\zeta$. Some combinations are impossible, because the monotonicities for $x$ and $\zeta$ conflict.
	\begin{equation}%
    \def\arraystretch{1.7}
	\label{eq:expTable}
	\begin{array}{r|cccc} 
	x\backslash\zeta & \overrightarrow{[0,\infty]}
	& \overleftarrow{[0,\infty]}
	& \overrightarrow{[-\infty, 0]}
	& \overleftarrow{[-\infty,0]}
	\\ \hline 
	\overrightarrow{[1,\infty]}
	& \overrightarrow{[1,\infty]}
	&&
	& \overleftarrow{[0,1]}
	\\
	\overleftarrow{[1,\infty]}
	&& \overleftarrow{[1,\infty]}
	& \overrightarrow{[0,1]} &
	\\
	\overrightarrow{[0,1]}
	&& \overrightarrow{[0,1]}
	& \overleftarrow{[1,\infty]} &
	\\
	\overleftarrow{[0,1]}
	& \overleftarrow{[0,1]}
	&&& \overrightarrow{[1,\infty]}
	\end{array}
	\end{equation}

	In Theorem~\ref{thm:zetaReal} we shall prove the monotone cases for $x\geq 1$ and $\zeta\geq 0$, top-left in the table, where $x$, $\zeta$ and $x^\zeta$ all have the same orientation.
	Meanwhile, however, it seems easiest to start with the case where $x\geq1$ is Dedekind, and $\zeta$ is signed: in Proposition~\ref{prop:onesidedrealexpI}, we lift $\zeta$ from rationals to one-sideds. In Theorem~\ref{thm:xzetaDed}, we extend this to get the case where $\zeta$ is Dedekind, and a gluing argument allows us to finally extend the construction to the whole space of positive Dedekinds. Further, since Theorem~\ref{thm:xzetaDed} also covers the case where $x$ is rational, we later use that in Theorem~\ref{thm:zetaDed} for cases where $x$ is one-sided.

	\subsection{Dedekind Real Base}\label{subsection:DedekindRealBASE} Once again, our plan of attack involves using the Lifting Lemma. 
	 Proposition~\ref{prop:monotoneantitone}, however, alerts us to the fact that the behaviour of $\{x^{q_n}\}_{n\in\mathbb{N}}$ differs depending on whether the base $0<x<1$ or $x>1$. Given the monotonicity condition of the Lifting Lemma, this indicates a natural case-splitting in our analysis.

	 As such, to control the behaviour of rational exponentials, we first restrict to the case when $x\geq 1$ --- this allows us to extend the range of $\zeta$ to the whole real line without worrying about monotonicity issues. We start by establishing the following two lemmas:

	\begin{lem}[Bernoulli's Inequality]\label{lem:bernoulli}
    For any positive real $x$ (Dedekind or one-sided), and any natural number $k\in\mathbb{N}$, we have the inequality:
		\[(1+x)^k \geq 1+ k\cdot x.\]
		\begin{proof} The proof is entirely algebraic, so it works identically regardless of whether $x$ is Dedekind or one-sided. We proceed by induction. For our base case $k=0$, we want to show:
			\[(1+x)^0\geq 1+ 0\cdot x = 1.\]
			But this is obvious since $(1+x)^0=1$. To prove the inductive hypothesis, suppose that the desired inequality holds for $k$. To show that it also holds for $k+1$, note that:
			\begin{align*}
			(1+x)^{k+1}&= (1+x)^k\cdot(1+x)\\
			&\geq (1+kx)\cdot (1+x)\\
			&= 1 + (k+1)\cdot x + k\cdot x^2\\
			&\geq 1+ (k+1)\cdot x,
			\end{align*}
			where the first inequality is by the inductive hypothesis, and the last inequality by the fact that multiplying two non-negative reals yields another non-negative real.
		\end{proof}

	\end{lem}

	\begin{lem}[Continuity Lemma]\label{lem:continuityI}
    Suppose $0<q<q'$, for a pair of (positive) rationals $q,q'\in Q_+$. Let $x$ be a Dedekind real such that $x\geq 1$. Then there exists a positive integer $k\in\mathbb{N}_+$ such that $q\leq q\cdot x^{\frac{1}{k}}<q'$.

	\end{lem}

	\begin{proof} Denote $\delta:=\frac{q'}{q}-1$ (which is a positive rational), and so $q'=q(1+\delta)$. By Bernoulli's Inequality and the Archimedean property, there exists $k\in\mathbb{N}_+$ such that:
		\[x<1+k\cdot \delta\leq (1+\delta)^k.\]
		By Lemma~\ref{lem:monotonepositiverationals}, this
		implies $1\leq x^{\frac{1}{k}}<1+\delta$, and so further multiplying through by $q$ yields:
		\[q\leq q\cdot x^{\frac{1}{k}}<q'. \qedhere\]
	\end{proof}

	Fixing a Dedekind $x\in\ropen{1,\infty}$, our definition of $x^\zeta$ (for arbitrary $\zeta\in\mathbb{R}$) rests on two levels of extensions. We first extend rational exponents to one-sided exponents, before combining the one-sided exponents to yield a Dedekind exponent.\footnote{Note that in the previous section (e.g.\ Proposition~\ref{prop:NatNumReals}), we fixed the exponent before applying lifting arguments to the base. In this setting, we work inversely: we fix the base before applying lifting arguments to the exponent.}

	\begin{prop}\label{prop:onesidedrealexpI}
		For Dedekind $x\geq 1$, the exponentiation by arbitrary rationals can be extended to one-sided exponents, giving exponentiation maps
		\[
		\ropen{1,\infty}\times\overrightarrow{[-\infty,\infty]} \to \overrightarrow{[0,\infty]}
		\text{ and }
		\ropen{1,\infty}\times\overleftarrow{[-\infty,\infty]} \to \overleftarrow{[0,\infty]}
		\text{.}
		\]
		Each map is unique subject to being monoid homomorphism for $\zeta\in \overrightarrow{[-\infty,\infty]}$ or  $\zeta\in \overleftarrow{[-\infty,\infty]}$,
		and satisfies the Basic Equations~\eqref{eq:BE1}--\eqref{eq:BE3}.
	\end{prop}
	\begin{proof} Fix Dedekind $x\in\ropen{1,\infty}$
		We prove the result for the lower case of $\zeta$. The upper case is analogous.

		Following Remark~\ref{rem:liftingvariation}, we check the two lifting conditions of the Lifting Lemma based on $\overrightarrow{[-\infty, \infty]}\cong\RIdl(\mathbb{Q}\cup\{-\infty\},<)$ (the involvement of $-\infty$ is largely irrelevant, so we shall leave this case to the reader). Monotonicity amounts to holding for any $q,r\in\mathbb{Q}$, we have that $q<r\implies x^q\leq x^r$ --- which is immediate from Proposition~\ref{prop:monotoneantitone}. For continuity: if $r<x^q$, with $r,q\in\mathbb{Q}$, we want rational $q' < q$ with $r < x^{q'}$. This is clear if $r\leq 0$, hence suppose instead that $0<r$. By strict order $<$, there exists $r'\in\mathbb{Q}$ such that $0<r<r'<x^q$. Applying the Continuity Lemma, we find $k$ with $rx^{1/k} < r'<x^q$. Then $r < x^{q-1/k}$.

		The basic equations follow from the case of rational exponents.
	\end{proof}

	\begin{thm}%
		\label{thm:xzetaDed}
		We have an exponentiation map on the Dedekinds
		\[
		\exp\colon (0,\infty)\times \Reals \to (0,\infty)
		\text{.}
		\]
		It satisfies the exponent laws, i.e.\ Basic Equations~\eqref{eq:BE1} and~\eqref{eq:BE2}.
	\end{thm}
	\begin{proof}First, we claim the two maps of Proposition~\ref{prop:onesidedrealexpI} combine to give an exponentiation map
		$\ropen{1,\infty}\times\Reals \to (0,\infty)$,
		by $x^\zeta=(x^{L_\zeta}, x^{R_\zeta})$.  From the definition, we calculate that:
		\[q<x^{\zeta} \Leftrightarrow  \exists q'\in \mathbb{Q}.(q'<\zeta\,\land \, q< x^{q'})\]
		\[q>x^{\zeta} \Leftrightarrow  \exists q'\in \mathbb{Q}.(q'>\zeta\,\land \, q>x^{q'}).\]
		As before, to show this map is well-defined, it remains for us to check inhabitedness, positiveness, separatedness and locatedness. Inhabitedness and positiveness are easy: we know there exist rationals $q_0,r_0\in\mathbb{Q}$ such that $q_0<\zeta<r_0$ (since $\zeta$ is an inhabited Dedekind) and so there exists $q,r\in\mathbb{Q}$ such that $q<x^{q_0}$ and $r>x^{r_0}$ (since $0<x^{q_0}$ and $x^{r_0}$ are inhabited as well).

		For separatedness, suppose $x^{R_\zeta}<q<x^{L_\zeta}$.
		Then we have $x^{r_1}<q<x^{r_2}$ for some rationals $r_1,r_2$ with $R_\zeta<r_1$ and $r_2<L_\zeta$. But $\zeta$ is separated, and so $r_2<r_1$, which implies $x^{r_2}<x^{r_1}$ by Proposition~\ref{prop:monotoneantitone}, a contradiction.

		For locatedness, suppose we have rationals $q<r$. When $q\leq 0$, then $q<x^{\zeta}$ since $x^{\zeta}$ is positive, so let's assume $q>0$. Leveraging previous results, we then define a series of parameters:
		\begin{itemize}
			\item Denote $r':=(q+r)/2$.
			\item By the Continuity Lemma, find $k$ such that $r'x^{1/k}<r$.
			\item By Remark~\ref{rem:locAxiom}, find a rational $s$ such that $s<\zeta<s+1/k$.
		\end{itemize}
		Since $x^s$ is a (located) Dedekind real, this means that $q<x^s$ or $x^s<r'$. If $q<x^s$ then $q<x^\zeta$, while if $x^s<r'$ then $x^\zeta<x^{s+1/k}=x^s\cdot x^{1/k}<r'\cdot x^{1/k}<r$.

		The exponent laws for this map follow from those for the maps in Proposition~\ref{prop:onesidedrealexpI}.

		Having defined $x^\zeta$ for $x\in\ropen{1,\infty}$, we can also define an exponentiation
		$\lopen{0,1}\times\Reals \to (0,\infty)$,
		by $x^\zeta:=(x^{-1})^{-\zeta}$. Let us now fix $\zeta\in\mathbb{R}$. Since the two maps agree on $x=1$, we can apply the Gluing Principle (Proposition~\ref{prop:gluingcoequaliser}) to glue them together and obtain the general exponentiation map via the pushout property:
		\[\begin{tikzcd}
		\{1\} \arrow[hook]{r} \arrow[hook]{d}& (0,1\text{]} \arrow[hook]{d}
		\arrow[ddr,bend left,"\argu^{\zeta}"] \\
		\text{[}1,\infty)\arrow[hook]{r}  \arrow[drr,bend right,swap,"\argu^{\zeta}"] & (0,\infty)
		\arrow[dr,dashed,"\argu^{\zeta}"] \\
		&& (0,\infty)
		\end{tikzcd}
        \]
		Externalising yields the desired exponentiation map $(0,\infty)\times\mathbb{R}\rightarrow (0,\infty)$ (cf. Convention~\ref{conv:fixingx}). The exponent laws follow immediately.
	\end{proof}

	\begin{rem} The reader may have noticed that we only proved the exponent laws in Theorem~\ref{thm:xzetaDed} --- this is because the base product equation ($(x\cdot y)^\zeta=x^\zeta \cdot y^\zeta$) does not transfer directly over the gluing, and thus requires separate proof. We defer proof of the base product law to Section~\ref{subsec:onesidedrealBase}.
	\end{rem}

	As an immediate corollary of the theorem, we generalise Proposition~\ref{prop:monotoneantitone} to obtain the following monotonic/antitonic result for real exponentiation with respect to the \emph{exponent}:

	\begin{cor}\label{cor:monotoneWRTexp} Let $\zeta,\zeta'\in\mathbb{R}$ such that $\zeta<\zeta'$. Then:
		\begin{enumerate}[label={(\roman*)}]
			\item  If $x\in(1,\infty)$ is a Dedekind real, then $x^{\zeta}<x^{\zeta'}$.
			\item if $x\in (0,1)$ is a Dedekind real, then $x^{\zeta'}<x^{\zeta}$.
		\end{enumerate}
	\end{cor}
	\begin{proof} Let $x\in (1,\infty)$, and $\zeta,\zeta'\in\mathbb{R}$ such that $\zeta<\zeta'$. By Exponent Law~\eqref{eq:BE1} established in Theorem~\ref{thm:xzetaDed}, we obtain: \[x^{\zeta'}=x^{\zeta'-\zeta+\zeta}=x^{\zeta'-\zeta}\cdot x^{\zeta}.\]
		We claim $x^{\zeta'-\zeta}>1$. Why?
		By strict order $<$ on the reals, pick some $r\in Q_+$ such that $0<r<\zeta'-\zeta$. Unpacking definitions, we know $q<x^{r}\implies q<x^{\zeta'-\zeta}$. By Proposition~\ref{prop:monotoneantitone}, we get $x^r>1$, which in turn implies that $x^{\zeta'-\zeta}>1$. In particular, we get:
		\[1<x^{\zeta'-\zeta}\implies x^{\zeta}<x^{\zeta'-\zeta}\cdot x^{\zeta}=x^{\zeta'}.\]
		The case when $x\in (0,1)$ is entirely analogous. \end{proof}

	\subsection{One-Sided Real Base}\label{subsec:onesidedrealBase} In contrast to the previous subsection, we now work with a one-sided real base and Dedekind exponent. Subtleties regarding negative exponents (as discussed at the start of this section) require some care, but they can be manoeuvred around sensibly.

	\begin{lem}%
		\label{lem:finalBasicRat}
		The base product law $(s\cdot t)^\zeta=s^\zeta \cdot t^\zeta$ holds for  $\zeta\in\Reals$ and \emph{positive rationals} $s,t\in Q_+$.
	\end{lem}
	\begin{proof} Similar to Propositions~\ref{prop:ratBE1} and~\ref{prop:ratBE2}, we shall need to case split based on which side of 1 the values $s,t,s\cdot t$ lie.
		If all three are at least 1, then the equation holds from Proposition~\ref{prop:onesidedrealexpI}. If at least one of them, say $s$, is less than 1, so $s^\zeta=(s^{-1})^{-\zeta}$, then the equation is equivalent to $(s^{-1})^\zeta\cdot  ((s\cdot t)^{-1})^\zeta=t^\zeta$, with $s^{-1}>1$.
		We may have to apply similar transformations for $t$ and $st$, but eventually we end up with an equation in which all three values are at least 1.
	\end{proof}

	To prove the base product law more generally, we shall work via the one-sided reals in our usual way.

	\begin{thm}%
		\label{thm:zetaDed}
		For Dedekind exponent $\zeta\geq 0$ and one-sided base $x$, we can define exponentiation maps
		\[
		\overrightarrow{\lopen{0,\infty}}\times \ropen{0, \infty} \to \overrightarrow{\lopen{0,\infty}}
		\text{ and }
		\overleftarrow{\ropen{0, \infty}}\times \ropen{0, \infty} \to \overleftarrow{\ropen{0, \infty}}
		\]
		such that, for $x$ Dedekind, $(L_x)^\zeta = L_{x^\zeta}$ and $(R_{x})^\zeta = R_{x^\zeta}$. The Basic Equations hold for these maps.
	\end{thm}
	\begin{proof} We prove the lower case. Following Remark~\ref{rem:gluingissues}, we know that a map from the rationals can be defined (geometrically) via case-splitting on $<$. Thus define:
		\begin{align*}
		Q_+\times \ropen{0, \infty}&\longrightarrow \overrightarrow{\lopen{0,\infty}}\\
		(s,\zeta)&\longmapsto \begin{cases}
		L_{s^{\zeta}}, \quad\,\,\qquad\qquad\qquad \text{if $s\geq 1$}\\
		(R_{(s^{-1})^\zeta})^{-1}, \qquad\qquad \text{if $s<1$}.
		\end{cases}
		\end{align*}
		Note that $(R_{(s^{-1})^\zeta})^{-1}$ is indeed a lower real since $\argu^{-1}\colon \overrightarrow{\lopen{0,1}}\cong\overleftarrow{\ropen{1,\infty}}$ (cf. Definition~\ref{def:inverses}).

		We now apply the Lifting Lemma in the case of $\overrightarrow{\lopen{0,\infty}}\cong \RIdl(Q_+,<)$. Monotonicity for $s<s'$ essentially follows from  Lemma~\ref{lem:finalBasicRat}, which yields:
		\[\left(\frac{s'}{s}\right)>1\implies (s')^\zeta\cdot (s^{-1})^{\zeta}=\left(\frac{s'}{s}\right)^{\zeta}\geq 1 \implies (s')^\zeta\geq s^\zeta.\]

		For continuity, suppose $q<s^\zeta$. If $1\leq s$, then by definition of $s^\zeta$ we have $0<q<s^r$ for some rational $r<\zeta$. Let $r=\frac{a}{b}$ for $a,b,\in\mathbb{N}_+$. Raising to the power of $b$, we know that  $q<s^{\frac{a}{b}}\iff q^b < s^a$ by Lemma~\ref{lem:monotonepositiverationals}. Applying Lemma~\ref{lem:rational sandwich}, there exists some positive rational $s'\in Q_+$ such that $q^b<(s')^a<s^a$, which (taking $b$th roots) yields the desired inequality $q<(s')^{r}<s^{r}<s^{\zeta}$.
		If instead $s<1$, then we have $s^\zeta=(s^{-1})^{-\zeta}$,
		and so $(s^{-1})^\zeta < q^{-1}$. By definition of  $(s^{-1})^\zeta$, there exists $r\in Q_+$ with $\zeta<r$ and $(s^{-1})^r<q^{-1}$. Applying Lemma~\ref{lem:monotonepositiverationals} again, we know there exists some $t$ such that  $s^{-1}<t<(q^{-1})^{\frac{1}{r}}$, and so $t^r<q^{-1}$.
	Thus, since $s^{-1}<t\iff t^{-1}<s$, by Corollary~\ref{cor:monotoneWRTexp} we get $q<(t^{-1})^r<(t^{-1})^\zeta<s^\zeta$.

		Finally, checking definitions, it is clear that $(L_{x})^{\zeta}=L_{x^{\zeta}}$. The proof for upper reals is similar\footnote{Note that the definition includes $0^0=\inf_{0<s\in\mathbb{Q}}s^0=1$.}.

		By Theorem~\ref{thm:xzetaDed} and Lemma~\ref{lem:finalBasicRat}, the basic laws hold for rational bases, and we can lift those to the one-sided bases.
	\end{proof}

	\begin{cor}%
		\label{cor:finalBasic}
		The map $\exp\colon (0,\infty)\times \mathbb{R}\rightarrow (0,\infty)$ of Theorem~\ref{thm:xzetaDed} satisfies the base product law.
	\end{cor}
	\begin{proof}
		Fixing $\zeta$, the base product equation is between two maps from $(0,\infty)^2$ to $(0,\infty)$. Since $(0,\infty)$ is locally compact, hence exponentiable, this amounts to equality between two maps $1 \to (0,\infty)^{(0,\infty)^2}$,
		and that is a subspace of 1 --- internally in $\mathcal{S}\mathbb{R}$.
		Hence the $\zeta$s for which that internal subspace is 1 form a subspace of $\mathbb{R}$. To prove that subspace is the whole of $\mathbb{R}$, we use the fact that $\mathbb{R}$ is the subspace join of $(-\infty, 0)$ and its closed complement $\ropen{0, \infty}$ (see, e.g.~\cite{StoneSp}, or, for a geometric treatment,~\cite{Vi2}). It thus  suffices to check the equation in the two cases $\zeta<0$ and $0\leq\zeta$.

		If $0\leq \zeta$, then the base product law follows from Theorem~\ref{thm:zetaDed}. If $\zeta<0$, then we reduce to the former case by applying the inverse map $\argu^{-1}$:
		\[(x\cdot y)^{\zeta}=\big((x\cdot y)^{-\zeta}\big)^{-1}=\big(x^{-\zeta}\cdot y^{-\zeta}\big)^{-1}=x^\zeta\cdot y^\zeta. \qedhere \]\end{proof}

	As an application of the base product law, we generalise Lemma~\ref{lem:monotonepositiverationals}, and establish monotonic/antitonic result for real exponentiation with respect to the \emph{base}:

	\begin{cor}\label{cor:monotoneWRTbase} Let $x,y\in (0,\infty)$ be positive Dedekind reals. Then:
		\begin{enumerate}[label={(\roman*)}]
			\item If $\zeta$ is a positive Dedekind real, then $x< y\implies x^{\zeta}< y^{\zeta}$.
			\item If $\zeta$ is a negative Dedekind real, then $x<y\implies y^{\zeta}<x^{\zeta}$.
		\end{enumerate}
	\end{cor}
	\begin{proof} Suppose $\zeta\in (0,\infty)$. Let us first prove (i) when $x=1$. In which case, pick some $q\in\mathbb{Q}$ such that $0<q<\zeta$. By Corollary~\ref{cor:monotoneWRTexp}, deduce that $1<y$ implies $1=1^{\zeta}<y^{q}<y^{\zeta}$, as desired.  In the general case when $x<y$ for $x,y\in(0,\infty)$, observe that $x<y\implies 1<y\cdot x^{-1}$. Apply Corollary~\ref{cor:finalBasic} to our previous calculation to get $1<(y\cdot x^{-1})^{\zeta}=y^{\zeta}\cdot x^{-\zeta}$, which in turn implies that $x^{\zeta}<y^{\zeta}$, proving (i). The case when $\zeta$ is negative follows from (i) and the fact that inverses reverese orientation.
	\end{proof}

	For completeness, we also deal with two entirely one-sided cases, as seen in table~\eqref{eq:expTable}.
	Proposition~\ref{prop:RationalExpReals} defines $x^q$ for $0<x$ a one-sided real and $0\leq q$ rational. Restricting to $1\leq x$, we now lift $q$ to one-sided $\zeta$.

	\begin{thm}%
		\label{thm:zetaReal}
		For one-sided $x\geq 1$, the exponentiation by non-negative rational exponents (Subsection~\ref{subsec:RationalExpReals}) can be extended to one-sided exponents, giving exponentiation maps
		\[
		\overrightarrow{[1,\infty]}\times\overrightarrow{[0,\infty]} \to \overrightarrow{[1,\infty]}
		\text{ and }
		\overleftarrow{[1,\infty]}\times\overleftarrow{[0,\infty]} \to \overleftarrow{[1,\infty]}
		\text{.}
		\]
		They satisfy the Basic Equations~\eqref{eq:BE1}--\eqref{eq:BE3}.
	\end{thm}
	\begin{proof}
		We prove the lower case. The upper case is very similar.

		Fix $x\in \overrightarrow{[1,\infty]}$, and recall that $Q$ denotes the non-negative rationals. Following Remark~\ref{rem:liftingvariation}, we apply the Lifting Lemma in the case of $\overrightarrow{[0,\infty]}\cong \RIdl(Q,<)$, with the special understanding that $0<0$.

		Monotonicity follows from Proposition~\ref{prop:monotoneantitone}. For continuity, suppose $r<x^q$ with $0< q$. We want to prove $r<x^{q'}$ for some $q'$ with $0< q'<q$. If $r<1$ then we can just take $q'=\frac{q}{2}$, so suppose $1\leq r$. Since $\argu^q$ preserves and reflects strict order, we may find some rational $s$ such that $1\leq r^{\frac{1}{q}}<s<x$, which in turn implies that $r<s^q<x^q$.
		Now using the fact that $s$, a rational, is Dedekind,
		we can apply the argument of Proposition~\ref{prop:onesidedrealexpI} to find our $q'$.

		The Basic Equations follow from the rational case.
	\end{proof}

	\section{Logarithms}\label{sec:log}
	Let $b\in(1,\infty)$ be Dedekind. Proposition~\ref{prop:onesidedrealexpI} and Theorem~\ref{thm:xzetaDed} then gives us three maps $b^\argu$ from $\mathbb{R}$ to $(0,\infty)$ for Dedekinds or, with adjustments for zero and infinities, for one-sided reals. The main result of this section is to prove that all three maps are invertible by constructing the relevant $\log_b$ maps. In the case of Dedekinds, this yields the following:
	\begin{align*}
	\log\colon (1,\infty) &\!\times\! (0, \infty) \longrightarrow\mathbb{R}\\
	(\,b\,&, \, y\,) \longmapsto \log_b(y)
	\end{align*}
	providing a geometric account of the usual logarithmic map. We remark that the case-splitting indicated by the monotone/antitone behaviour of rational exponents $x^q$ when $x< 1$ vs. $x> 1$ emerges here through the case-splitting of the logarithmic base $b$, which forces us to consider the case of $b\in (0,1)$ and $b\in(1,\infty)$ separately.

	Before proceeding, we shall need the following lemma.

	\begin{lem}\label{lem:logarithmcts} Let $b\in (1,\infty)$ be Dedekind, and $q,q'\in Q_+$ be positive rationals such that $0<q<q'$. Then there exists rational $r$ such that $q<b^{r}<q'$.
	\end{lem}
	\begin{proof}[Proof of Lemma] This is a sharper version of the Continuity Lemma~\ref{lem:continuityI}, which only required the Dedekind base to satisfy $b\geq 1$. In contrast, this lemma requires the Dedekind base $b$ to be \textit{strictly} greater than $1$ (it is clearly false when $b=1$).

		We start by defining a series of parameters:
		\begin{itemize}
			\item By Bernoulli's Inequality (Lemma~\ref{lem:bernoulli}), find a natural number $t$ such that $q' < b^t$, and a negative integer $s$ such that $b^s < q$
			(i.e.\ $q^{-1} < b^{-s}$). 
			\item By the Continuity Lemma, find $k\in\mathbb{N}_+$ such that $1\leq b^{\frac{1}{k}}<1+\frac{q'-q}{b^t}$.

		\end{itemize}
		Observe that $s+\frac{i}{k}\leq t$, for all $i\in\mathbb{N}$ such that $0\leq i\leq (t-s)k$, with equality exactly when $i=(t-s)k$. This yields:
		\[ b^{s}\cdot b^{\frac{(i+1)}{k}}-b^{s}\cdot b^{\frac{i}{k}}=b^{s+\frac{i}{k}}(b^{\frac{1}{k}}-1)<b^{t}\bigg(\frac{q'-q}{b^t}\bigg)=q'-q.\]

		By Corollary~\ref{cor:monotoneWRTexp}, $\big\{b^{s+\frac{i}{k}}\big\}_{0\leq i\leq (t-s)k}$ is a monotonic sequence of \textit{located} Dedekinds. In particular, this means $q<b^{s+\frac{i}{k}}\vee b^{s+\frac{i}{k}}<q'$ for all relevant $i$. Thus, by a reasoning similar to the proof of Lemma~\ref{lem:rational sandwich}, one shows there indeed exists an $i\in\mathbb{N}_+$ so that $q< b^{s+\frac{i}{k}}<q'$. \end{proof}

	\begin{thm}\label{thm:logarithmsI} Fix $b\in (1,\infty)$. We can then define one-sided logarithm maps
		\[\log_b\colon \overrightarrow{[0,\infty]}\rightarrow \overrightarrow{[-\infty,\infty]} \quad \text{and} \quad \log_b\colon \overleftarrow{[0,\infty]}\rightarrow \overleftarrow{[-\infty,\infty]}\]
		inverse to the corresponding exponentiation maps $b^{\argu}$ on the one-sideds. These combine to yield an isomorphism on the Dedekinds
		\[\log_b\colon (0,\infty)\xrightarrow{\sim}(-\infty,\infty)\]
	\end{thm}
	\begin{proof} The proof proceeds in stages.
	\paragraph{Step 0: Set-up}	Fix $b\in (1,\infty)$. In the case where $y$ is a one-sided real, we define the $\log_b$ maps as:
		\[q<\log_b(y) \Leftrightarrow  b^q<y\]
		\[q>\log_b(y)\Leftrightarrow b^q>y \]
		with the understanding that when $y$ is a lower real, we define $b^q<y$ to mean $b^q<r<y$ for some rational $r$ (and similarly when $y$ is upper). Note that this definition makes sense since $b^q$ is Dedekind. Finally, we remark that we shall freely make use of Corollary~\ref{cor:monotoneWRTexp} without explicit mention, in particular the fact that $q<r \implies b^{q}<b^r$ for any rationals $q$ and $r$.
\paragraph{Step 1: $\log_b(y)$ as a one-sided}	To verify downward closure and upper-roundedness for the lower reals, observe that this coincides with verifying the monotonicity and continuity conditions for the Lifting Lemma in our proof of Proposition~\ref{prop:onesidedrealexpI}. We remark that when $y=0$, then $\log_b0$ is the empty lower real, i.e.\ $-\infty$. The case for upper reals is similar.

\paragraph{Step 2: $\log_b(y)$ as a Dedekind} After Step 1, it remains to check inhabitedness, separatedness and locatedness. Inhabitedness can be derived from Lemma~\ref{lem:logarithmcts}.

		To show separatedness, suppose $q<\log_b(y)<r$. By definition, this means $b^q<y<b^r$. By decidability of $<$ on $\mathbb{Q}$, we know that $r\leq q$ or $q<r$. If $r\leq q$, then this yields $b^r\leq b^{q}<b^{r}$, a  contradiction. Hence, this forces the inequality $q<r$.

		For locatedness, suppose we have $q,r\in\mathbb{Q}$ such that $q<r$. As before, we know that $q<r$ yields the inequality $b^{q}<b^{r}$. Pick $u_1,u_2\in\mathbb{Q}$ such that $b^{q}<u_1<u_2<b^{r}$. Since $y$ is located, either $u_1<y$ or $y<u_2$. Since $u_1<y\implies b^{q}<y$ and $y<u_2\implies y<b^{r}$, this implies $q<\log_b(y)\vee r>\log_b(y)$, i.e.\ $\log_b(y)$ is located.

\paragraph{Step 3: $\log_b$ and $b^{\argu}$ are inverses} We prove $\log_b$ and $b^{\argu}$ are inverses on the one-sideds, where
		\[b^{\argu}\colon \overrightarrow{[-\infty,\infty]}  \rightarrow\overrightarrow{[0,\infty]}\quad \text{and} \quad b^{\argu}\colon\overleftarrow{[-\infty,\infty]}\rightarrow \overleftarrow{[0,\infty]}. \]
		The case for the Dedekinds follows immediately by Corollary~\ref{cor:check1sided}.

\paragraph{Step 3a: Show that $b^{\log_{b}(y)}=y$.}
Now suppose $q<b^{\log_b(y)}$. Then by definition (cf. Theorem~\ref{thm:xzetaDed}), $\exists q'\in\mathbb{Q}$ such that $q'<\log_b(y)$ and $q<b^{q'}$. Since $q'<\log_b(y)\implies b^{q'}<y$, this assembles to yield the inequality $q<b^{q'}<y$, proving that $b^{\log_b(y)} \leq y$. Conversely, suppose $q<y$. If $q<0$, then $q<b^{\log_b(y)}$ automatically. If $q\geq 0$, then pick $q',q''\in Q_+$ such that $q<q'<q''<y$. By Lemma~\ref{lem:logarithmcts}, there exists $r\in\mathbb{Q}$ such that $q<q'<b^r<q''<y$, hence $r<\log_b(y)$ and $q<b^{\log_b (y)}$, i.e.\ $y\leq b^{\log_b(y)}$. Since $y\leq b^{\log_b(y)}$ and $ b^{\log_b(y)}\leq y$, this shows that $b^{\log_b(y)}=y$ for lower reals. The case for the upper reals is entirely analogous.
\paragraph{Step 3b: Show that $\log_{b}(b^\zeta)=\zeta$}  Suppose $q<\log_b(b^\zeta)$. Then, unwinding definitions, there exists rationals $q',q''\in\mathbb{Q}$ such that $b^q<q'<b^{\zeta}$ and $q''<\zeta\land q'<b^{q''}$. In particular, this yields the inequality $b^{q}<b^{q''}$, which in turn implies $q<q''<\zeta$, proving that $\log_b(b^{\zeta})\leq \zeta$. Conversely, suppose that $q<\zeta$. A routine calculation shows that $b^{q}<b^{\zeta}$, which by definition yields $q<\log_b(b^{\zeta})$, proving that $\zeta\leq \log_b(b^\zeta)$. Putting everything together yields $\log_{b}(b^\zeta)=\zeta$. As before, the case for upper reals is entirely analogous.
	\end{proof}

	\begin{obs} As a sanity check, note that Steps 3a and 3b in the previous proof verify standard logarithmic identities. In particular, for the map $\log_b\colon (0,\infty)\rightarrow (-\infty,\infty)$ on the Dedekinds, Step 3 essentially says: ``Given any $y\in(0,\infty)$, $\log_b(y)$ is the \textit{unique} Dedekind $\zeta$ such that $b^{\zeta}=y$ holds.'' In particular, suppose $y=b$. Then since $b^{1}=b$, uniqueness of $\log_b(b)$ implies that $\log_b(b)=1$, as expected. Similarly, suppose $y=1$. Since $b^0=1$, this thus implies that $\log_b(1)=0$, again exactly as expected.
	\end{obs}

	Theorem~\ref{thm:logarithmsI} fixes $b\in(1,\infty)$ in order to define a map $\log_b \argu \colon (0,\infty)\rightarrow(-\infty,\infty)$ on the Dedekinds. Externalising this map yields the desired map
	\[\log\colon(1,\infty)\times (0,\infty)\rightarrow \mathbb{R}.\]

	\begin{rem} One can also define a logarithm map with base $b\in (0,1)$ in terms of the previous $\log$ map, as follows:
		\begin{align*}
		\log\colon(0,1)&\times (0,\infty)\rightarrow \mathbb{R}\\
		(\, b &,y\, ) \longmapsto -\log(\,{b^{-1}},y\,).
		\end{align*}
	\end{rem}

	\section{Comparisons with other approaches and Future Applications  }\label{sec:ExpNT}

	\renewcommand{\thesubsection}{\arabic{section}.\arabic{subsection}}

	\subsection{Comparisons with Other Possible Approaches} It is reasonable to ask: why did we choose to develop point-free real exponentiation in the way we did? For instance, why did we choose to work with Dedekind reals as opposed to Cauchy reals? Or why did we not choose to define exponentiation via power series --- e.g.\ by first defining the functions
	\[\exp(x):= \displaystyle\sum^{\infty}_{n=0}\frac{x^n}{n!} \quad \quad \text{for $x\in\mathbb{R}$}\]
	\[\displaystyle\ln(x):=\int^{x}_1 t^{-1} dt \quad\quad\text{for $x\in (0,\infty$)},\]
	before defining $x^\zeta\!:= \exp \,( \zeta \ln (x))$, for $x\in (0,\infty)$ and $\zeta\in\mathbb{R}$? The answer to both questions is that we felt that our chosen approach would be comparatively cleaner in the point-free setting. Indeed, whilst there does exist a point-free account of quotiented Cauchy reals~\cite[Theorem 7.8]{Vi6} and of integration~\cite{Vi7}, both descriptions involve rather complicated and technical details. In contrast, by choosing to define real exponentiation by successively lifting it from the rational case to the reals, our construction explicitly (and geometrically) highlights how many familiar algebraic properties of real exponentiation are in fact inherited from the properties of $\mathbb{Q}$ as one might expect --- something which would likely to have been obscured in the other two approaches.

	One may also reasonably ask: why even construct a point-free account of real exponentiation in the first place? One simple answer is that point-free topology is a school of constructive mathematics with many attractive features, and so it's worth translating familiar ideas from real analysis into this setting and demonstrating how they (more or less) work in the ways we expect them to. From a methodological standpoint, it is also instructive to illustrate how point-wise reasoning works in a point-free setting. A slightly deeper answer, however, comes from our interest in applying ideas from geometric logic to questions arising from number theory, which we explain more fully in the next subsection.

	\subsection{Applications to Number Theory and Arithmetic Geometry} Suppose we have some object $X$ defined over $\mathbb{Q}$. In number theory, a result is often said to follow Hasse's Local-Global Principle just in case some property $P$ holds for $X$ over $\mathbb{Q}$ iff $P$ holds for $X$ over all non-trivial completions of $\mathbb{Q}$ (up to topological equivalence). The classic examples concern the existence of non-trivial rational solutions to polynomials. To illustrate, consider some polynomial equation,
 	say:
	\[X^3 + Y^3 + Z^3 = 0.\]
One may then ask: is it true that there exists a non-trivial rational solution [i.e.\ rationals $a,b,c\in\mathbb{Q}$ such that $a^3+b^3+c^3=0$ and $a,b,c$ are not all 0] iff there exists a non-trivial solution over all the non-trivial completions of $\mathbb{Q}$?

Given an arbitrary polynomial, the problem of determining the local conditions for the existence of rational solutions (if any) is very hard in full generality. There are well-known cases when Hasse's Principle does indeed hold (e.g.\ quadratic forms over $\mathbb{Q}$) but also when it fails (e.g.\ Selmer's famous counter-example of $3x^3+4y^3+5z^3=0$), and we are still some ways off from a complete understanding of this arithmetic phenomena. Nevertheless, Hasse's Principle illustrates something very interesting: it highlights a close relationship between all the different completions of $\mathbb{Q}$, and how this might be leveraged to obtain deep number-theoretic results.

One would thus like to improve our understanding of the completions of $\mathbb{Q}$. Here, however, important subtleties begin to emerge. By Ostrowski's Theorem, we know that all non-trivial completions of $\mathbb{Q}$ are either equivalent to the $p$-adic fields $\mathbb{Q}_p$ (for some prime $p$) or the field of real numbers $\mathbb{R}$. It is well-established that there are many deep disanalogies between the $p$-adics and the reals, which often results in the development of technologies that work well for one but not the other\footnote{For this reason, many often choose to work with just the $p$-adics (e.g.\ via the finite adele ring $\mathbb{A}^{\text{fin}}_{\mathbb{Q}}$) and ignore the reals. See, for instance Balchin-Greenlees' work~\cite{BG} on Adelic Models for tensor-triangulated categories, or Huber's work~\cite{Huber} on the Beilinson-Parshin adeles, where she writes: ``We want to stress that at this stage only a generalization of the finite adeles is found. It is not clear what one should take at infinity, or in fact even what the infinite `places' should be.''}. As Mazur muses~\cite{Maz}:

\begin{quote}``A major theme in the development of Number Theory has been to try to bring $\mathbb{R}$ somewhat more into line with the $p$-adic fields; a major mystery is why $\mathbb{R}$ resists this attempt so strenuously.''
\end{quote}

On this note, let's return to point-free topology for a moment. Suppose there exists a geometric theory $\mathbb{T}$ that axiomatises the completions of $\mathbb{Q}$ up to topological equivalence. By Observation~\ref{obs:conserv}, we know $[\mathbb{T}]$ possesses a generic point $U_\mathbb{T}$ which is conservative in the sense that any geometric property $P$ is satisfied by $U_\mathbb{T}$ iff $P$ is satisfied by all points of $[\mathbb{T}]$, i.e.\ all the completions of $\mathbb{Q}$ up to topological equivalence. Thus, if such a $[\mathbb{T}]$ exists, then this gives a (geometric) framework for reasoning about \textit{all} completions of $\mathbb{Q}$ via the the generic completion of $\mathbb{Q}$. This opens up many natural questions of number-theoretic interest, in particular how geometricity or generic reasoning might provide new insights into this subtle relationship between the $p$-adic fields $\mathbb{Q}_p$ and $\mathbb{R}$.

	The first step towards constructing this topos is understanding when a given completion $K$ of $\mathbb{Q}$ is topologically equivalent to another completion $K'$. Classically, completions of $\mathbb{Q}$ are defined as point-set spaces comprising the Cauchy sequences of $\mathbb{Q}$ with respect to some kind of metric on $\mathbb{Q}$:
	\begin{align*}
	|\cdot|\colon\mathbb{Q}&\longrightarrow \ropen{0, \infty}\\
	x&\longmapsto |x|
	\end{align*}
	known as an \textit{absolute value}. Given two absolute values $|\cdot|_1,|\cdot|_2$, we can define an equivalence relation $\sim$ where $|\cdot|_1\sim|\cdot|_2$ iff there exists some $\alpha\in\lopen{0,1}$ such that $|x|^{\alpha}_1=|x|_2$ or $|x|^{\alpha}_2=|x|_1$ for all $x\in\mathbb{Q}$ such that $x\neq 0$. Such an equivalence class of absolute values is called a \textit{place}, and it turns out that two absolute values belong to the same place iff their completions are topologically-equivalent. Hence, before defining the topos of completions of $\mathbb{Q}$, we would first like to define the classifying topos of places of $\mathbb{Q}$. In order to do this, we need to be able to express the equivalence relation geometrically, which in turn requires a geometric account of real exponentiation --- which was the very content of this paper. In forthcoming work, we shall construct and investigate the properties\footnote{Morally speaking, there ought to be some kind of connection between the classifying topos of places of $\mathbb{Q}$ and the (hypothesised) Arakelov compactification of $\Spec(\mathbb{Z})$. Analysing this topos thus serves as an excellent test problem for bringing into focus how certain essential aspects of the topology and algebra interact with one another, which ought to give us some insight into the deeper mechanics underlying the number theory and homotopy theory.} of the topos of places of $\mathbb{Q}$.

	\section*{Acknowledgements} The authors would like to thank Thomas Streicher for challenging us with this problem --- not only to solve it, but also to think carefully about the potential applications of its solution. We also thank the referees for their thoughtful comments.

	\nocite{BB}
	\nocite{MM}
	\nocite{Vi1}
    \bibliographystyle{alphaurl}
	\bibliography{Exponentiation} 

\end{document}